\documentclass{article}

\usepackage{amsmath}
\usepackage{amssymb}
\usepackage{graphicx}
\usepackage{tikz}
\usepackage{bm}
\usepackage[linesnumbered,ruled]{algorithm2e}
\usepackage{mathdots}
\usepackage{booktabs}
\usepackage{rotating}
\usepackage{cases}
 \usepackage{bigints}
\usepackage[ruled,linesnumbered]{algorithm2e}
\usepackage[a4paper,left=2.8cm,right=2.8cm,top=2.5cm,bottom=2.5cm]{geometry}
\usepackage{fancyhdr}
\usepackage[framemethod=tikz]{mdframed}
\newtheorem{theorem}{Theorem}[section]
\newtheorem{corollary}{Corollary}[section]
\newtheorem{proposition}{Proposition}[section]
\newtheorem{lemma}{Lemma}[section]
\newtheorem{definition}{Definition}[section]

\newtheorem{remark}{Remark}[section]

\makeatletter 
\@addtoreset{equation}{section}
\makeatother  

\newenvironment{proof}{{\noindent\it Proof.}\quad}{\hfill $\square$\\}

\newcommand{\circled}[1]{\tikz[baseline=(char.base)]{
            \node[shape=circle,draw,inner sep=0.7pt] (char) {#1};}}

\usepackage{url}
\usepackage{mathtools}

\begin{document}
\title{Parameter choice strategies for regularized least squares approximation of noisy continuous functions on the unit circle}
\author{Congpei An\footnotemark[1],\footnotemark[2]
       \quad \text{and}\quad Mou Cai\footnotemark[3],\footnotemark[4]}
\renewcommand{\thefootnote}{\fnsymbol{footnote}}
\footnotetext[1]{School of Mathematics and Statistics, Guizhou University, Guiyang 550025, Guizhou, China (andbachcp@gmail.com)}
\footnotetext[2]{School of Engineering, University of Tokyo, 7-3-1 Hongo, Bunkyo-Ku, Tokyo (caimoumou@g.ecc.u-tokyo.ac.jp, caimou1230@163.com)}
\footnotetext[4]{Corresponding author}
\maketitle

\begin{abstract}
This paper explores the incorporation of Tikhonov regularization into the least squares approximation scheme using trigonometric polynomials on the unit circle. This approach encompasses interpolation and hyperinterpolation as specific cases. With the aid of the de la Vallée-Poussin approximation, we derive a uniform error bound and a concrete $L_2$ error bound. These error estimates demonstrate the effectiveness of Tikhonov regularization in the denoising process. A new regularity condition for the selection of regularization parameters is proposed. We investigate three strategies for choosing regularization parameters: Morozov's discrepancy principle, the L-curve, and generalized cross-validation, by explicitly combining these error bounds of the approximating trigonometric polynomial. We show that Morozov's discrepancy principle satisfies the proposed regularity condition, while the other two methods do not.
Finally, numerical examples are provided to illustrate how the aforementioned methodologies, when applied with well-chosen parameters, can significantly improve the quality of approximation.
\end{abstract}


\section{Introduction}\label{sec;introduction}
Periodic functions are essential in various fields, including science and engineering, and it is crucial to recover these functions approximately from observational data, often affected by noise. For instance, four commonly used periodic functions in digital synthesis are the square wave, sawtooth wave, triangular wave, and sine wave \cite{AN2023115, Stilson1996AliasFreeDS}. When approximating periodic functions, algebraic polynomials are unsuitable since they are not periodic. Therefore, we focus on approximating functions using trigonometric polynomials over the unit circle $\mathbb{S}^1 := \{(\cos\theta, \sin\theta)^T \mid \theta \in [-\pi, \pi]\}$. Our goal is to find an appropriate parameter for regularized trigonometric polynomials based on $N$ different sampling points on $\mathbb{S}^1$ to effectively approximate $f \in C(\mathbb{S}^1)$, where $C(\mathbb{S}^1)$ represents the space of continuous functions on $\mathbb{S}^1$.

In practice, the sampling procedure is often contaminated by noise, leading to an ill-posed problem when trying to obtain a good approximation from perturbed Fourier expansions \cite[Section 1.2]{lu2013regularization}. To address this, we can employ regularization techniques. Previous works \cite{GenH1979, An_2020, Anlasso2021, Hansen1993} have utilized Tikhonov regularization to mitigate noise, adding an $\ell_2^2$ penalty to the classical least-squares strategy \cite{tikhonov1977solutions}.

This paper focuses on finding a class of trigonometric polynomial approximations of $f \in C(\mathbb{S}^1)$ that emerge as minimizers of the Tikhonov regularized least squares problem formulated as follows:
\begin{align}\label{DisLeatsquares}
\min_{p \in \mathbb{P}_L} \left\{ \sum_{j=1}^N w_j(p(x_j) - f^\epsilon(x_j))^2 + \lambda \sum_{j=1}^N (\mathbb{\mathcal{R}}_L p(x_j))^2 \right\},
\end{align}
where $f^{\epsilon}(x_j) := f(x_j) + \epsilon_j$ denotes the noisy values of a contaminated version $f^{\epsilon}$ of the original function $f$ at the points $\mathcal{X}_N = \{x_1, x_2, \ldots, x_N\} \subset \mathbb{S}^1$. The weights $w_j$ are specified, $\mathbb{P}_L := \mathbb{P}_L(\mathbb{S}^1)$ is the polynomial space on the circle of degree $\leq L$, and 
\[
\mathbb{\mathcal{R}}_L : \mathbb{P}_L \rightarrow \mathbb{P}_L
\]
is a linear "penalization" operator that can be defined in various ways, with $\lambda > 0$ representing the regularization parameter. We can choose the regularization operator $\mathbb{\mathcal{R}}_L$ in its most general rotationally invariant form, expressed as:

\begin{align}\label{rotationinvary}
\mathcal{R}_L p(x) = \sum_{\ell=0}^L \beta_{\ell,1} \int_{-\pi}^\pi p(x) \cos(\ell x) \, dx + \sum_{\ell=1}^L \beta_{\ell,2} \int_{-\pi}^\pi p(x) \sin(\ell x) \, dx,
\end{align}
where $\beta_{0,1}, \beta_{1,1}, \beta_{1,2}, \beta_{2,1}, \beta_{2,2}, \ldots, \beta_{L,1}, \beta_{L,2}$ form a nondecreasing sequence of nonnegative real parameters. Throughout this paper, we assume the point set $\mathcal{X}_N$ consists of points dictated by the trapezoidal rule:

\begin{align}\label{exactness}
\int_{-\pi}^\pi p(x) \, dx = \frac{2\pi}{N} \sum_{j=1}^{N} p(x_j), \quad \forall p \in \mathbb{P}_{N-1},
\end{align}
with $x_j, j = 1, 2, \ldots, N$ representing an equidistant subdivision on $\mathbb{S}^1$ with a step size of $2\pi/N$. The periodicity condition $f(x_0) = f(x_N)$ is also maintained. It is well known that when the integrand is a periodic function with certain smooth properties, the approximation via the trapezoidal rule often exhibits exponential convergence \cite{ExpTrapezoidal}.

In particular, we can construct entry-wise closed form solutions to problem \eqref{DisLeatsquares} with the help of trapezoidal rules. Consequently, we obtain a \textit{regularized barycentric trigonometric interpolation} with a special parameter setting. Comparing this computational approach to traditional trigonometric interpolation, it is both faster and more stable. Moreover, we examine the approximation quality of problem \eqref{DisLeatsquares} in terms of the uniform norm and the \(L_2\) norm, providing guidance for choosing parameters in trigonometric approximation schemes. Determining a proper regularization parameter in a judicious way is the central task of this paper.

Numerical experiments in \cite{Anchen2012}, \cite{An_2020} illustrate that a proper choice of regularization parameter \(\lambda\) can significantly improve the approximation quality. However, \cite{An_2020} does not specify a strategy for parameter choice. Hesse and Le Gia \cite{L2errorestimates} adopt Morozov’s discrepancy principle \cite{Morozov1966OnTS} to find a suitable parameter \(\lambda\) based on an \(\ell_2^{2}\) penalty model over the sphere. Although Morozov’s discrepancy principle is simple and stable for finding a good parameter \(\lambda\), it requires prior knowledge of the noise level. As a complement, we also consider the L-curve method \cite{Hansen1993}, which is an algorithm to find the corner of a log-log plot. Generalized cross-validation \cite{GenH1979}, a popular statistical method, calculates a regularization parameter without noise level information. We study these three methods to determine the corresponding ``optimal" parameter in our model \eqref{DisLeatsquares}. 

Due to the explicit expression of the approximation schemes, these parameter selection calculations become more direct and convenient, avoiding the iterative calculations and linear algebra problems found in previous research \cite{GenH1979}, \cite{Hansen1993}, \cite{Morozov1966OnTS}. Therefore, this research method enables a better evaluation of the advantages and disadvantages of these three parameter selection methods, thereby achieving improved approximation results for recovering functions from noisy data. Notably, we propose Definition 5.1 as a novel criterion for parameter selection in regularized approximation problems.

This paper is organized as follows. In the next section, we present the necessary preliminaries. In Section 3, we provide the solution to the Tikhonov regularized least squares problem in closed form and transform the solution into the form of regularized barycentric trigonometric interpolation. In Section 4, we theoretically analyze the error bounds of the approximation \(f \approx p\) in terms of the \(L_2\) norm and the uniform norm. In Section 5, we discuss parameter choice strategies concerning the penalization parameters \(\beta_{\ell,1}, \beta_{\ell,2}\) and the regularization parameter \(\lambda\). Finally, we present some numerical experiments that test the theoretical results from the previous sections in Section 6, concluding with remarks in Section 7.

\section{Preliminaries}
Let $x,y\in\mathbb{S}^1$. Denoting that $x=(\cos\theta,\sin\theta)^T$, $y=(\cos\phi, \sin\phi)^T$,
$x\cdot y=\cos(\theta-\phi)$.
We introduce spherical harmonics \footnote{The spherical harmonics are specialized terms for depicting the basis of spherical polynomial space $\mathbb{S}^d \;(d\geq 1).$ Although $\mathbb{S}^1$ is just a circle, we still call the basis of circular polynomial space $\mathbb{S}^1$ by spherical harmonics for consistency with $\mathbb{S}^d.$} for spherical polynomial space 
\begin{align*}
 \mathbb{P}_L:=\text{span} \left\{Y_{0,1},\;Y_{\ell,k}(x),   \;\;\;k=1,2,\;\ell=1,\ldots,L  \right\}, \;\;x\in\mathbb{S}^1,
\end{align*}
where 
\begin{align*}
Y_{0,1}:=1/\sqrt{2\pi},\;Y_{\ell,1}(x) := \cos(\ell \theta)/\sqrt{\pi},\;Y_{\ell,2}(x) :=\sin(\ell \theta)/\sqrt{\pi}.
\end{align*}Noting that the dimension of $\mathbb{P}_L$ is $d:=\dim{(\mathbb{P}_L)}=2L+1$.
The spherical harmonics $Y_{\ell,k}$ are assumed to be orthonormal with respect to the standard $L_2$ inner product,
 \begin{align}\label{continuousinner}
\left<f,g \right>_{L_2} :=\int_{-\pi}^{\pi} f g \;\;dx.
\end{align}
Then for $p\in \mathbb{P}_L$ an arbitrary circular polynomial of degree $\leq L,$ there exists a unique vector $\boldsymbol{\alpha}=(\alpha_{0,1},\alpha_{1,1},\alpha_{1,2},\ldots,\alpha_{L,1},\alpha_{L,2})^T \in \mathbb{R}^{2L+1}$  such that
$$p(x)=\sum_{\ell=0}^L \sum_{k=1}^{\gamma+1} \alpha_{\ell,k} Y_{\ell,k}(x),\; x\in\mathbb{S}^1,   $$
where
$$ \gamma = \left\{
\begin{array}{ll}
0\;\; & \text{if } \;\ell = 0, \\
1 \;\;& \text{else }. 
\end{array}
\right. $$
The addition theorem for spherical harmonics $Y_{\ell,k}$ \cite{AtkinsonHan} on $\mathbb{S}^1,$ which asserts
\begin{align}
\sum_{\ell=0}^L \sum_{k=1}^{\gamma+1} Y_{\ell,k}(x)Y_{\ell,k}(y)&=\frac{1}{2\pi} \sum_{\ell=-L}^L   \exp(i\ell(\theta-\phi)) \label{reproduceKernel}\\
                                                                &=\frac{1}{2\pi} \frac{\sin( (L+1/2)(\theta-\phi))}{\sin((\theta-\phi)/2)}. \nonumber
\end{align}

Let \( G_L(x,y) := \sum_{\ell=0}^L \sum_{k=1}^{\gamma+1} Y_{\ell,k}(x) Y_{\ell,k}(y) \). One can verify that \( G_L(x,y) \) is a reproducing kernel \cite{Reproducing_Kernels, pereverzyev2022introduction}, which suggests that
\begin{flushleft}
\circled{1}\quad $G_L(x,\cdot) \in \mathbb{P}_L,$\\ [10pt]
\circled{2}\quad $G_L(x,y)=G_L(y,x),$\\ [10pt]
\circled{3}\quad $\begin{aligned}[t]
\left<p(x),G_L(x,y)  \right>_{L_2}&=\left<\sum_{\ell=0}^L \sum_{k=1}^{\gamma+1} \alpha_{\ell,k} Y_{\ell,k}(x),\;\sum_{\ell=0}^L \sum_{k=1}^{\gamma+1} Y_{\ell,k}(x)Y_{\ell,k}(y)  \right>_{L_2} \\
&=\sum_{\ell=0}^L \sum_{k=1}^{\gamma+1} \alpha_{\ell,k} Y_{\ell,k}(y)=p(y),\; p\in\mathbb{P}_L,
\end{aligned}$
\end{flushleft}
where the second equality is from the orthogonality of spherical harmonics $Y_{\ell,k}.$

The trapezoidal rule is often exponentially accurate for periodic continuous functions \cite{ExtPeriodic2015}, which would play an important role in this paper. Thanks to the exactness of trapezoidal rule \eqref{exactness},  we define a ``discrete inner product''  which is first introduced by Sloan \cite{SLOAN1995238}
\begin{align}\label{discreteinner}
\left<v,z \right>_N :=\frac{2\pi}{N} \sum_{j=1}^N v(x_j) z(x_j),
\end{align}
corresponding to the ``continuous'' inner product \eqref{continuousinner}. Analogously to the trigonometric projection, Sloan gives the corresponding discrete inversion, named hyperinterpolation on the unit circle $\mathbb{S}^1$ \cite[Example B]{SLOAN1995238}, which is expressed as the following definition.
\begin{definition}

  For $N\geq 2L+1$, the \textit{hyperinterpolation} of $f\in \mathcal{C}(\mathbb{S}^1)$ onto $\mathbb{P}_L$ is defined as 
\begin{align} \label{hypercircle}
\mathcal{L}_L f:=\sum_{\ell=0}^L \sum_{k=1}^{\gamma+1} \left<f,Y_{\ell,k} \right>_N Y_{\ell,k}.
\end{align}

\end{definition}

The additional smoothness of $f$ is determined by the convergence rate of its Fourier coefficients. In \cite{lu2013regularization,Parameterchoice2015}, this rate is measured by the $smoothness \;index\; function\;\phi(t),$ a prescribed continuous nondecreasing function, which is defined by $ \phi: [a,b]\rightarrow [0,\infty ],\;\phi(0)=0,$ tends to zero as $t\rightarrow 0.$ Then standard assumption on the smoothness of a function $f$ is expressed in terms of a Hilbert space $W^{\phi,\beta},$ namely,
\begin{align*}
f\in W^{\phi,\beta} := \left\{g: \Vert g\Vert_{W^{\phi,\beta}} :=\sum_{\ell=0}^\infty \sum_{k=1}^{\gamma+1} \frac{|\left< g, Y_{\ell,k}(x)\right>_{L_2}|^2}{\phi^2(\beta_{\ell,k}^{-2}  )}   <\infty\right\},
\end{align*}
where $\boldsymbol{\beta}=\left\{\beta_{0,1},\beta_{1,1},\beta_{1,2},\ldots,\beta_{L,1},\beta_{L,2} \right\}$ is the sequence of coefficients appearing in the regularizer \eqref{rotationinvary}.

We denote by $ f^\epsilon $ a noisy version $f,$ and regard both $f$ and $f^\epsilon$ as continuous for the following analysis. See, for example, \cite{Parameterchoice2015},   let $\epsilon=\left[\epsilon_1,\;\epsilon_2,\ldots,\epsilon_N \right]^{T} \in \mathbb{R}^N$ and $\Vert \epsilon\Vert_\infty=\max |\epsilon_i|.$ Then, it is natural to assume that $\Vert f-f^\epsilon\Vert_\infty=\Vert \epsilon\Vert_\infty,$ which allows the worst noise level at any point of $\mathcal{X}_N.$

The error of best approximation of $f\in\mathcal{C}(\mathbb{S}^1)$  by  an element $p$ of $\mathbb{P}_L$ is also involved, which is defined by
\begin{align}
E_L(f):=\inf_{p\in \mathbb{P}_L} \Vert f-p\Vert_\infty.
\end{align}
As is known \cite{Jackson1911} that $E_L(f) \rightarrow 0$, as $\;L\rightarrow \infty.$

 \section{Construction of solutions }
  \subsection{Regularized least squares approximation}
 For convenience, we just consider the even length of the interval, i.e., $N$ is an odd positive integer. The function $f$ sampled on $\mathcal{X}_{N}$ generates
 \begin{align*}
 \mathbf{f}:=\mathbf{f}(\mathcal{X}_N)=\left[f(x_1),f(x_2),\ldots,f(x_N) \right]^T \in \mathbb{R}^{N},
 \end{align*}
and note that the rotationally invariant operator defined by \eqref{rotationinvary} satisfies
\begin{align*}
\mathcal{R}_L Y_{\ell,k}=\beta_{\ell,k} Y_{\ell,k}\; \;\; {\rm for}\; \;Y_{\ell,k}(x)\in \mathbb{P}_L,\;\;k=1,2,\;\ell=0,1,\ldots,L.
\end{align*}
Let $\mathbf{A}:=\mathbf{A}(\mathcal{X}_N)\in \mathbb{R}^{N \times (2L+1)}$  be consist of spherical harmonics $Y_{\ell,k}(x)$ evaluated at the points set $\mathcal{X}_N$. 
$$
\mathbf{A}=\begin{pmatrix}
Y_{0,1}(x_1) & Y_{1,1}(x_1)     & Y_{1,2}(x_1)    & \dots &   Y_{L,1}(x_1)   &   Y_{L,2}(x_1)                 \\
Y_{0,1}(x_2) & Y_{1,1}(x_2)     & Y_{1,2}(x_2)     & \dots  &   Y_{L,1}(x_2)   &   Y_{L,2}(x_2)                                   \\
\vdots       & \vdots            & \vdots             & \ddots&    \vdots &    \vdots                           \\
Y_{0,1}(x_N) & Y_{1,1}(x_N)      & Y_{1,2}(x_N)    &\dots  & Y_{L,1}(x_N)    & Y_{L,2}(x_N) 
\end{pmatrix}.
$$
Then the problem \eqref{DisLeatsquares} can be transformed into the discrete regularized least squares problem
\begin{align}\label{matrixLeast}
\min_{\boldsymbol{\alpha} \in \mathbb{R}^{2L+1}} \Vert \mathbf{W}^{1/2} \left(\mathbf{A} \boldsymbol{\alpha}-\mathbf{f} \right)\Vert_2^2 +\lambda \Vert  \mathbf{W}^{1/2}\mathbf{R}_L^{\mathbf{T}}\boldsymbol{\alpha} \Vert_2^2,\;\;\lambda>0,
\end{align}
where $\mathbf{R}_L:=\mathbf{R}_L(\mathcal{X}_N)=\mathbf{B}_L \mathbf{A^T}  \in \mathbb{R}^{(2L+1)\times N}$  with $\mathbf{B}_L$ a positive semidefinite diagonal matrix defined by

\begin{align*}
\mathbf{B}_L:= {\rm diag} (\beta_{0},\beta_{1,1},\beta_{1,2},\ldots,\beta_{L,1},\beta_{L,2} ) \in \mathbb{R}^{(2L+1)\times (2L+1)},
\end{align*}
and $$\mathbf{W}:= {\rm diag} (2\pi/N,\ldots,2\pi/N)\in\mathbb{R}^{N\times N}.$$
Taking the first derivative of  \eqref{matrixLeast}, we could obtain the following system of linear equations:
\begin{align}\label{firstorder}
 \left(\mathbf{A^T W A} +\lambda \mathbf{B}_L \mathbf{A^T W A}\mathbf{B}_L^{\mathbf{T}} \right)\boldsymbol{\alpha}=\mathbf{A^T W f},\qquad \lambda>0.
\end{align}
When we take the equidistant points on unit circle, the solution to the first order condition \eqref{firstorder} can be obtained in an entry-wise closed form with the special structure of $\bf{A^T W A},$ which is just an identity matrix.

\begin{lemma}\label{section2:lemma1}Assume that the point set $ \mathcal{X}_{N}=\{x_1,\ldots,x_{N}\}$ satisfies the \textit{trapezoidal rule} (see \eqref{exactness}) with
   $2L+1\leq N.$ Then 
  \begin{align}\label{identity}
  \mathbf{A^T W A}=\mathbf{I} \in \mathbb{R}^{(2L+1)\times (2L+1)}.
  \end{align}
\end{lemma}

\begin{proof}
By the structure of $\bf{A^T W A},$ since the exactness property \eqref{exactness} of trapezoidal formula, we obtain
$$
[ \mathbf{A^T W A}]_{\ell,\ell',k,j}=\frac{2\pi}{N}\sum_{j=1}^{N} Y_{\ell,k}(x_j)Y_{\ell',k}(x_j)=\int_{-\pi}^{\pi} Y_{\ell,k}(x) Y_{\ell',k}(x)\; {\rm dx}=\delta_{\ell,\ell'},
$$
where $\ell,\ell' =0,1,\ldots,L$ and $j=1,2,\ldots,N,\;k=1,2,$ and $\delta_{\ell,\ell'}$ is the Kronecker delta. The second equality holds from $Y_{\ell,k}(x) Y_{\ell',k}(x) \in \mathbb{P}_{2L} \subseteq \mathbb{P}_N(\mathbb{S}^1), $ and the last equality holds because of the orthogonality of $ Y_{\ell,k}(x).$
\end{proof}

Consequently, we obtain the following interesting result.

\begin{theorem}\label{illcondition}
Under the conditions of Lemma \ref{section2:lemma1}, suppose \(\mathbf{A^T W f} \neq \mathbf{0}\). The condition number of the linear system \eqref{firstorder},
\[
{\rm cond}_2 (\mathbf{G}) := \Vert \mathbf{G} \Vert_2 \Vert \mathbf{G}^{-1} \Vert_2,
\]
is monotonically increasing for the parameter \(\lambda > 0\).
\end{theorem}
\begin{proof}
Since \(\mathbf{A^T W A}\) is the identity matrix, we can simplify the coefficient matrix in \eqref{firstorder} as follows:

\begin{align*}
\mathbf{G} = \mathbf{A^T W A} + \lambda \mathbf{B}_L \mathbf{A^T W A} \mathbf{B}_L^{\mathbf{T}} = \mathbf{I} + \lambda \mathbf{B}_L \mathbf{B}_L^{\mathbf{T}}.
\end{align*}

It is clear that \(\mathbf{G}\) is diagonal. Therefore, the condition number of \(\mathbf{G}\) is given by

\begin{align}\label{condition}
{\rm cond}_2 (\mathbf{G}) = \frac{1 + \lambda \beta_{\max}^2}{1 + \lambda \beta_{\min}^2},
\end{align}

where \(\beta_{\max}\) and \(\beta_{\min}\) denote the maximum and minimum values of \(\beta_{\ell,k}\), respectively. 

Next, we compute the derivative:

\begin{align*}
\frac{d}{d \lambda} \left( \frac{1 + \lambda \beta_{\max}^2}{1 + \lambda \beta_{\min}^2} \right) = \frac{\beta_{\max}^2 - \beta_{\min}^2}{(1 + \lambda \beta_{\min}^2)^2} \geq 0.
\end{align*}

Thus, we conclude that the condition number is monotonically increasing with respect to \(\lambda\), completing the proof.
\end{proof}

\begin{remark}
From equation \eqref{condition}, we can also know that the increasing rate of penalization parameter $\beta_{\ell,k}$ would affect the condition number for the linear system \eqref{firstorder}. The more rapidly the rate increases, the larger the condition number.  
\end{remark}

\begin{theorem}\label{theorem1}
Under the conditions of Lemma \ref{section2:lemma1}.  For $f\in C(\mathbb{S}^1)$, the optimal solution to problem \eqref{matrixLeast} can be expressed by
 \begin{equation}\label{coe}
\alpha_{\ell,k}= \frac{\left<f,Y_{\ell,k} \right>_N}{1+\lambda \beta_{\ell,k}^2} ,\;\;\;\; k=1,2;\;\ell=0,1,\ldots, L. \\
 \end{equation}
Consequently, the minimizer to the problem \eqref{DisLeatsquares} is
\begin{equation}\label{soltutionp}
p^{\beta}_{\lambda,L,N}(x)=\sum_{\ell=0}^{L}\sum_{k=1}^{1+\gamma}  \frac{\left<f,Y_{\ell,k} \right>_N}{1+\lambda \beta_{\ell,k}^2} Y_{\ell,k}(x) .
\end{equation}
\end{theorem}
\begin{proof}
This is immediately obtained from the first order condition \eqref{firstorder} of the problem \eqref{matrixLeast} and Lemma \ref{section2:lemma1}.
\end{proof}
\begin{remark}
When $\lambda=0,$ coefficients \eqref{coe} reduce to $ \left<f,Y_{\ell,k} \right>_N$. Then
\begin{equation}\label{triginter}
p_{0,L,N}(x)=\sum_{\ell=0}^{L}\sum_{k=1}^{1+\gamma}  \left<f,Y_{\ell,k} \right>_N Y_{\ell,k}(x),
\end{equation}
which can be regarded as the hyperinterpolation \cite{2006Hyperinterpolation,SLOAN1995238} on the unit circle \eqref{hypercircle}. Besides, if $2L+1=N,$ \eqref{triginter} is also a trigonometric interpolation on an equidistant grid \cite{zygmund2002trigonometric}, i.e.,
\begin{align}\label{intercondition}
p_{0,L,N}(x_j)=f(x_j),\;\;j=1,\ldots,N.
\end{align}
\end{remark}

\begin{remark}
The assumption $\mathbf{A^T W f}\neq \mathbf{0}$ is important, which can ensure the coefficients $\alpha_{\ell,k},\;k=1,2,\;\ell=0,1,\ldots,L$ are not all zero, as we want to avoid the case that approximation nonzero function $f$ by zero polynomial.
\end{remark}

\subsection{Regularized barycentric trigonometric interpolation}
We emphasize a novel trigonometric polynomial in this subsection that could lead to a revolution in trigonometric polynomial approximation. This specific trigonometric polynomial can be obtained under the interpolation condition from \eqref{soltutionp} using barycentric trigonometric interpolation and a constant parameter \(\beta_{\ell,k}\).

Barycentric trigonometric interpolation was introduced by Salzer \cite{Salzer1948CoefficientsFF} and later simplified by Henrici \cite{Henrici1979} and Berrut \cite{Berrut1984}. It gained popularity through the work of Berrut and Trefethen \cite{BaryInterpolation2004}. The interpolation takes the interesting form:

\begin{align}\label{bary}
t_{N}(x) &= \frac{ \sum_{j=1}^N  (-1)^j f(x_j) \, {\rm csc} \; \frac{x-x_j}{2}} { \sum_{j=1}^N  (-1)^j \, {\rm csc} \; \frac{x-x_j}{2}}.
\end{align}
Here, \({\rm csc} \; x := (\sin x)^{-1}\). The term ``barycentric" derives from its formal similarity to the formulas for the center of gravity (barycenter) of a system of masses \( (-1)^j \, {\rm csc} \; \frac{x-x_j}{2} \) attached to the points \( f(x_j) \). 

Compared with classical trigonometric interpolation, computing \eqref{bary} only requires \( O(N) \) operations (including evaluations of trigonometric functions). Moreover, \eqref{bary} is stable in most practical cases of interest \cite{Xu_stability2016, Berrut1984, Henrici1979}. Based on their studies, we propose regularized barycentric trigonometric interpolation, which introduces a multiplicative correction constant into the barycentric trigonometric interpolation formula.

\begin{theorem}
Let $2L+1=N.$ Adopt conditions in Lemma \ref{section2:lemma1}. Setting $\beta_{\ell,k}=\tau \geq 0,$ for all $\ell,\;k.$ Then  \eqref{soltutionp} can be transformed into a regularized barycentric trigonometric interpolation formula (when $N$ is an odd number).
\begin{align}\label{trigbarycentric}
t_{N}^{{\rm reg}}(x)=p_{\lambda,L,N}^\beta (x)=\frac{1}{1+\lambda \tau} \frac{ \sum_{j=1}^N (-1)^j f(x_j) {\rm csc} \frac{x-x_j}{2}}{\sum_{j=1}^N (-1)^j  {\rm csc} \frac{x-x_j}{2}}.
\end{align}

\end{theorem}
\begin{proof}
The approximation trigonometric polynomial \eqref{soltutionp} can be written as
\begin{align*}
p_{\lambda,L,N}^\beta (x)=\frac{2\pi}{(1+\lambda \tau)N} \sum_{j=1}^N f(x_j) \sum_{\ell=0}^L \sum_{k=1}^{1+\gamma}  Y_{\ell,k}(x_j) Y_{\ell,k}(x).
\end{align*}
Then we have a sum of a geometric sequence as follows
\begin{align*}
2\pi \sum_{\ell=0}^L \sum_{k=1}^{1+\gamma}  Y_{\ell,k}(x_j) Y_{\ell,k}(x)&=  \sum_{\ell=-L}^L \exp(i \ell(x-x_j)) \\
&= \frac{[1-\exp(iN(x-x_j))]\exp(-iL(x-x_j))}{\exp(ix)\left[\exp(-ix)-\exp(-ix_j)  \right]}.
\end{align*}
By $\exp(i\alpha)-\exp(i\beta)=2i \exp(\frac{i(\alpha+\beta)}{2}) \sin(\frac{\alpha-\beta}{2})$ and $\exp(\frac{i N}{2}x_j)=i(-1)^j,$ we have
\begin{align*}
2\pi \sum_{\ell=0}^L \sum_{k=1}^{1+\gamma}  Y_{\ell,k}(x_j) Y_{\ell,k}(x)&= \frac{\exp(-\frac{i N}{2}x)(-1)^j- \exp(\frac{i N}{2}x)(-1)^j}{2 \sin(\frac{x-x_j}{2})}. \\
\end{align*}
As a matter of fact, the constant function $f(x)  \equiv  1$ has trigonometric interpolation
\begin{align*}
1=\frac{2\pi}{N}\sum_{\ell=0}^L \sum_{k=1}^{1+\gamma} Y_{\ell,k}(x_j) Y_{\ell,k}(x),
\end{align*}
and cancel the common factor in $p_{\lambda,L,N}^\beta (x)/1,$ we obtain \eqref{trigbarycentric}.
\end{proof}

This new trigonometric polynomial \eqref{trigbarycentric} represents a periodic variant of Tikhonov-regularized barycentric interpolation \cite[Theorem 3.1]{An_2020}. As demonstrated in that work, \eqref{trigbarycentric}, when combined with barycentric trigonometric interpolation and a suitable correction constant, retains the computational efficiency and numerical stability of its classical counterparts. Moreover, it inherits desirable properties from Tikhonov regularization, such as improved robustness to noise.

\section{Error Analysis}

In this section we estimate the error of approximation of $f$ by $p^{\beta}_{\lambda,L,N}(x)$ in terms of $L_2$ norm and uniform norm $ \Vert f\Vert_\infty=\max_{x\in\mathbb{S}^1} |f(x)|$  in the presence of noise. 
\subsection{$L_2$ error}
The approximation trigonometric polynomial \eqref{soltutionp} can be deemed as an operator $\mathcal{U}_{\lambda,L,N}^{\beta}$ act on $f(x),$ thus 
\begin{align}\label{U_lambda}
 \mathcal{U}_{\lambda,L,N}^{\beta} f(x):=p^\beta_{\lambda,L,N}(x)= \sum_{\ell=0}^L \sum_{k=1}^{1+\gamma} \frac{1}{1+\lambda \beta_{\ell,k}^2}\left<f,Y_{\ell,k} \right>_N Y_{\ell,k}(x).
 \end{align}
The  $L_2$ norm of the operator is defined by
\begin{align*}
\Vert  \mathcal{U}_{\lambda,L,N}^{\beta} \Vert_{L_2} : =\sup_{f\neq 0} \frac{\Vert \mathcal{U}_{\lambda,L,N}^{\beta} f\Vert_{L_2}}{\Vert f\Vert_\infty}=\sup_{f\neq 0}\frac{\Vert  p^\beta_{\lambda,L,N} \Vert_{L_2}}{\Vert f\Vert_\infty}.
\end{align*}
It is clear that $\Vert  \mathcal{U}_{\lambda,L,N}^\beta \Vert_{L_2}\rightarrow   \Vert \mathcal{U}_{0,L,N} \Vert_{L_2}$ as $\lambda \rightarrow \infty.$ When $\lambda=0,$ the approximation polynomial reduces to the hyperinterpolation on the unit circle defined by \eqref{hypercircle}:
\begin{align} \label{U_0}
  \mathcal{L}_L f(x) :=\mathcal{U}_{0,L,N} f(x)=\sum_{\ell=0}^L \sum_{k=1}^{1+\gamma} \left<f,Y_{\ell,k} \right>_N Y_{\ell,k}(x).
\end{align}
As is shown in \cite{SLOAN1995238}, 
\begin{align}\label{hyperL2error}
 \left \Vert \mathcal{U}_{0,L,N} f \right \Vert_{L_2} \leq \sqrt{2\pi} \Vert f\Vert_\infty.
\end{align}
However, our regularized approximation operator $\mathcal{U}_{\lambda,L,N}^\beta$ would enlarge the $L_2$ norm compared with the classical version. At first, we suppose that the first penalization parameter $\beta_{0,1}$ is zero.

\begin{proposition}
Suppose $2L+1 \leq N$ and $\beta_{0,1}=0.$   Let $\mathcal{U}_{\lambda,L,N}^{\beta} f$ be defined by \eqref{U_lambda}. Then
\begin{align*}
\Vert \mathcal{U}_{\lambda,L,N}^\beta f\Vert_{L_2}\leq \sqrt{2\pi}  C(\lambda,\beta) \Vert f\Vert_\infty,
\end{align*}
where 
\begin{align}\label{Clambda}
C(\lambda,\beta)=\sqrt{1+\frac{1}{\lambda^2}\sum_{\ell=1}^L \sum_{k=1}^{2} \frac{1}{\beta_{\ell,k}^4}}.
\end{align}
\end{proposition}
\begin{proof}
From the definition of \eqref{U_lambda}, we have
\begin{align*}
\Vert \mathcal{U}_{\lambda,L,N}^\beta f\Vert_{L_2}^2&=\left|\left<f,Y_{0,1} \right>_N \right|^2+\sum_{\ell=1}^L \sum_{k=1}^{2} \frac{\left|\left<f,Y_{\ell,k} \right>_N \right|^2}{\left(1+\lambda \beta_{\ell,k}^2\right)^2}\\
&\leq  \left|\left<f,Y_{0,1} \right>_N \right|^2+\frac{1}{\lambda^2}\sum_{\ell=1}^L \sum_{k=1}^{2} \frac{\left|\left<f,Y_{\ell,k} \right>_N \right|^2}{\beta_{\ell,k}^4}.
\end{align*}
By using the Cauchy-Schwarz inequality, we obtain 
\begin{align*}
\left|\left<f,Y_{\ell,k} \right>_N \right|^2&=\frac{4\pi^2}{N^2}\left|\sum_{j=1}^N f(x_j)Y_{\ell,k}(x_j) \right|^2\\
&\leq \frac{2\pi}{N} \left(\sum_{j=1}^N |f(x_j)|^2\right) \left(\frac{2\pi}{N}\sum_{j=1}^N  Y_{\ell,k}^2(x_j)\right)\\
&=\frac{2\pi}{N} \left(\sum_{j=1}^N |f(x_j)|^2\right) \Vert Y_{\ell,k}(x)\Vert_{L_2}\\
&\leq 2\pi \Vert f\Vert_\infty^2,
\end{align*}
where the last second equality is due to the exactness of the trapezoidal rule \eqref{exactness}. Thus the bound
\begin{align*}
\Vert \mathcal{U}_{\lambda,L,N}^\beta f\Vert_{L_2}\leq \sqrt{2\pi}\Vert f\Vert_\infty \sqrt{1+\frac{1}{\lambda^2}\sum_{\ell=1}^L \sum_{k=1}^{2} \frac{1}{\beta_{\ell,k}^4}}.
\end{align*}
is obtained, which is the required result.
\end{proof}

Here we introduce a useful analysis tool named da la Vall\'{e}e-Poussin approximation \cite[Chapter 9]{devore1993constructive} $V_n f $ defined by  
\begin{align}
V_n f(x) =\sum_{\ell=0}^{2n-1} \sum_{k=1}^{1+\gamma} h(\ell,k) \left<f, Y_{\ell,k} \right>_{L_2} Y_{\ell,k}(x)\in \mathbb{P}_n, \label{davaliapproxi}
\end{align}
where $n=\lfloor \frac{L}{2} \rfloor$ (In this paper, without any special explanation, we set all $n=\lfloor \frac{L}{2} \rfloor.$), $h$ is a filter function which satisfies
$$ h(\ell,k)=\left\{
\begin{aligned}
&1 \;\;\;\;\;\; \;\;\;\;\;\;\;\;\;\;\; | \ell| \leq n,\\
& 1-\frac{\ell-n}{n}\;\;\;\;  n+1\leq |\ell|\leq 2n-1.
\end{aligned}
\right.
$$ 
Inspired by \cite{2011Polynomial}, we obtain that
\[ V_n p =p\;\; \text{for}\; \text{all}\; p \in \mathbb{P}_n,\]
\[\left<V_nf,Y_{\ell,k} \right>_{L_2}=h( \ell,k) \left<f,Y_{\ell,k} \right>_{L_2},\quad\ell=0,1,\ldots,L.\] And the norm of $V_n $ is bounded, that is
$\Vert V_n f\Vert_\infty\leq3\Vert f\Vert_\infty.$ Moreover
\begin{align*}
\Vert V_n f- f \Vert_{\infty} \leq (1+\Vert V_n\Vert_\infty) \Vert f-p^*\Vert_\infty=4E_n(f).
\end{align*}
With this bound, we could also make $L$ large enough to ensure $\Vert V_n f-f \Vert_\infty \leq \Vert \epsilon\Vert_\infty,$ as $E_L \rightarrow 0,\;L\rightarrow \infty.$  At first,  we estimate the $L_2$ error between approximation trigonometric polynomial \eqref{soltutionp} and hyperinterpolation \eqref{hypercircle} for $V_n f,$ which is useful for our next error analysis. In our next analysis, we always assume that the smoothness index function $\phi(t)$ is defined by $ \phi(t)=t^{1/2}.$
\begin{lemma} \label{lamma:lambdahypererror}
 Suppose $2L+1 \leq N$ and $\beta_{0,1}=0.$   Then
\begin{align}\label{hyperL2_unifromhelp}
\left \Vert \left(\mathcal{U}_{0,L,N}- \mathcal{U}_{\lambda,L,N}^{\beta} \right)V_n f \right\Vert_{L_2} \leq \frac{\sqrt{\lambda}}{2}\Vert f\Vert_{W^{\phi,\beta}}.
\end{align}
\end{lemma}
\begin{proof}
With the exactness of trapezoidal rule \eqref{exactness} and connection $$\left<V_nf,Y_{\ell,k} \right>_{L_2}=h( \ell,k) \left<f,Y_{\ell,k} \right>_{L_2},$$ we may write
\begin{align*}
\left \Vert \left(\mathcal{U}_{0,L,N}- \mathcal{U}_{\lambda,L,N}^{\beta} \right) V_n f \right\Vert_{L_2}&=\left \Vert \sum_{\ell=1}^L \sum_{k=1}^{2} \frac{\lambda \beta_{\ell,k}^2}{1+\lambda \beta_{\ell,k}^2} \left<V_n f, Y_{\ell,k} \right>_{L_2} Y_{\ell,k}(x) \right\Vert_{L_2} \\
&=\left \Vert \sum_{\ell=1}^L \sum_{k=1}^{2} h\left(\ell,k\right) \frac{\lambda \beta_{\ell,k}^2}{1+\lambda \beta_{\ell,k}^2}   \left<f, Y_{\ell,k} \right>_{L_2} Y_{\ell,k}(x) \right\Vert_{L_2} \\
    &=\left(\sum_{\ell=1}^L \sum_{k=1}^{2} h^2\left(\ell,k\right) \left(\frac{\lambda \beta_{\ell,k}^2}{1+\lambda \beta_{\ell,k}^2} \right)^2 \left| \left<f,Y_{\ell,k} \right>_{L_2} \right|^2     \right)^{1/2}.
\end{align*}
The third equality is from Parseval's equality. Thus, we have
\begin{align*}
\left \Vert \left(\mathcal{U}_{0,L,N}- \mathcal{U}_{\lambda,L,N}^{\beta} \right)V_n f \right\Vert_{L_2}&\leq \left(\sum_{\ell=1}^L \sum_{k=1}^{2} \left(\frac{\lambda \beta_{\ell,k}^2}{1+\lambda \beta_{\ell,k}^2} \right)^2 \phi^2(\beta_{\ell,k}^{-2}) \frac{\left| \left<f,Y_{\ell,k} \right>_{L_2} \right|^2}{ \phi^2(\beta_{\ell,k}^{-2})}     \right)^{1/2} \\
&=\left(\sum_{\ell=1}^L \sum_{k=1}^{2} \frac{\lambda^2 }{\beta_{\ell,k}^2(1/\beta_{\ell,k}^2+\lambda)^2}   \frac{\left| \left<f,Y_{\ell,k} \right>_{L_2} \right|^2}{ \phi^2(\beta_{\ell,k}^{-2})}     \right)^{1/2} \\
&\leq \sup_{u\in \left(0, \beta_{1,1 }^{-1} \right] } \left| \frac{\lambda u}{\lambda+u^2} \right| \Vert f\Vert_{W^{\phi,\beta}}\\&
\leq \frac{\sqrt{\lambda}} {2}\Vert f\Vert_{W^{\phi,\beta}}.
\end{align*}
\end{proof}
Now we are going to estimate the $L_2$ regularization error 
 $\Vert \mathcal{U}_{\lambda,L,N}^{\beta} f-f \Vert_{L_2}.$
\begin{theorem}\label{L_2 error}
Suppose $2L+1 \leq N$ and $\beta_{0,1}=0.$ Given noisy version $f^\epsilon \in \mathcal{C}(\mathbb{S}^1)$ .  Then
\begin{align} \label{L2erroranalysis}
\Vert \mathcal{U}_{\lambda,L,N}^\beta f^\epsilon -f\Vert_{L_2} \leq \sqrt {2\pi}\bigg {(}2C(\lambda,\beta)+1 \bigg{)}\Vert \epsilon\Vert_\infty +\frac{\sqrt{\lambda}}{2}\Vert f\Vert_{W^{\phi,\beta}},
\end{align}
where $C(\lambda,\beta)$ is defined by \eqref{Clambda}.
\end{theorem}
\begin{proof}
For any $V_n f \in \mathbb{P}_n,$ we have
\begin{align*}
\Vert \mathcal{U}_{\lambda,L,N}^\beta f^\epsilon -f \Vert_{L_2} &\leq \Vert \mathcal{U}_{\lambda,L,N}^\beta (f^\epsilon-V_n f)\Vert_{L_2}+\Vert (\mathcal{U}_{0,L,N}-\mathcal{U}_{\lambda,L,N}^\beta) V_n f \Vert_{L_2}+\Vert f- V_n f \Vert_{L_2} \\
&\leq   \sqrt{2\pi} C(\lambda,\beta) \Vert f^\epsilon-V_n f\Vert_\infty+\frac{\sqrt{\lambda}}{2}\Vert f\Vert_{W^{\phi,\beta}}+\sqrt{2\pi}\Vert f-V_n f\Vert_\infty,
\end{align*}
and by the decomposition $\Vert f^\epsilon-V_n f\Vert_\infty \leq \Vert f^\epsilon-f\Vert_\infty+\Vert f-V_n f\Vert_\infty,$  we obtain \eqref{L2erroranalysis}.
\end{proof}

\subsection{Uniform error}
To estimate the uniform error $\Vert \mathcal{U}_{\lambda,L,N}^{\beta} f-f \Vert_{\infty},$ we start with the Lebesgue constant of the operator $\mathcal{U}_{\lambda,L,N}^{\beta},$ which can be defined as follows. 
\begin{align*}
\Vert  \mathcal{U}_{\lambda,L,N}^{\beta} \Vert_\infty : =\sup_{f\neq 0} \frac{\Vert \mathcal{U}_{\lambda,L,N}^{\beta} f\Vert_\infty}{\Vert f\Vert_\infty}=\sup_{f\neq 0}\frac{\Vert  p^\beta_{\lambda,L,N} \Vert_\infty}{\Vert f\Vert_\infty}.
\end{align*}
Consequently, we have
\begin{theorem}
Adopt conditions of Lemma \ref{lamma:lambdahypererror}. Then
\begin{align}
\Vert \mathcal{U}_{\lambda,L,N}^{\beta}  \Vert_\infty\leq 1+\sum_{\ell=1}^L \sum_{k=1}^{2} \frac{\sqrt{2}}{1+\lambda \beta_{\ell,k}^2}.\label{lebesgue}
\end{align}
\end{theorem}
\begin{proof}
By the definition of the Lebesgue constant of  $\mathcal{U}_{\lambda,L,N}^\beta $ and using the Cauchy-Schwartz inequality, we have
\begin{align}
\Vert \mathcal{U}_{\lambda,L,N}^{\beta}  \Vert_\infty&=\max_{x\in\mathbb{S}^1} \sum_{j=1}^N \frac{2\pi}{N} \left| \sum_{\ell=1}^L \left( \sum_{k=1}^{2} \frac{1}{1+\lambda \beta_{\ell,k}^2} Y_{\ell,k}(x_j) Y_{\ell,k}(x)+Y_{0,1}(x_j)Y_{0,1}(x)  \right)\right| \nonumber \\
 &\leq 
 1+ \max_{x\in\mathbb{S}^1}\frac{2\pi}{N} \sum_{j=1}^N  \sum_{\ell=1}^L \sum_{k=1}^{2} \frac{1}{1+\lambda \beta_{\ell,k}^2} \left| Y_{\ell,k}(x_j) Y_{\ell,k}(x)  \right| \nonumber\\
 &\leq  1+ \max_{x\in\mathbb{S}^1}\sum_{\ell=1}^L \sum_{k=1}^{2} \frac{1}{1+\lambda \beta_{\ell,k}^2} \left(\frac{2\pi}{N} \sum_{j=1}^N Y_{\ell,k}^2(x_j)\right)^{\frac{1}{2}} \left(\frac{2\pi}{N}\sum_{j=1}^N  Y_{\ell,k}^2(x)\right)^{\frac{1}{2}} \nonumber \\
 &\leq 1+\sum_{\ell=1}^L \sum_{k=1}^{2} \frac{\sqrt{2}}{1+\lambda \beta_{\ell,k}^2}. \nonumber
\end{align}
\end{proof}

Inspired by \cite[Theorem 4.2]{Parameterchoice2015}, we have Theorem \ref{lebesgeuconstant} as following.\\
\begin{theorem}\label{lebesgeuconstant}
Suppose $2L+1 \leq N.$  Given a noisy version $f^\epsilon \in \mathcal{C}(\mathbb{S}^1)$ . Let $\mathcal{U}_{\lambda,L,N}^{\beta} f \in \mathbb{P}_n$ be defined by \eqref{U_lambda}. Then
\begin{align} \label{inferroranalysis}
\Vert \mathcal{U}_{\lambda,L,N}^\beta f^\epsilon -f\Vert_\infty \leq c \Vert \epsilon\Vert_\infty \big{\Vert} \mathcal{U}_{\lambda,L,N}^\beta \big{\Vert} _\infty +\sqrt{\lambda L} \Vert f\Vert_{W^{\phi,\beta}},
\end{align}
where $c$ is a generic constant.
\end{theorem}
\begin{proof}
For $V_n f \in \mathbb{P}_n,$ we have
\begin{align}
\Vert \mathcal{U}_{\lambda,L,N}^{\beta} f-f \Vert_{\infty} &=\Vert \mathcal{U}_{\lambda,L,N}^{\beta} (f-V_n f+V_n f)-f-V_n f+V_n f \Vert_{\infty} \nonumber \\
                                                          &\leq \Vert \mathcal{U}_{\lambda,L,N}^{\beta}(f-V_n f)\Vert_{\infty}+\Vert f-V_n f\Vert_{\infty}+\left \Vert \left(\mathcal{U}_{0,L,N}- \mathcal{U}_{\lambda,L,N}^{\beta} \right)V_n f \right\Vert_{\infty} \nonumber\\
                                                          &\leq \big{\Vert} \mathcal{U}_{\lambda,L,N}^{\beta}  \big{\Vert}_\infty \Vert f-V_n f\Vert_\infty+\Vert f-V_n f\Vert_\infty+\left\Vert \left(\mathcal{U}_{0,L,N}- \mathcal{U}_{\lambda,L,N}^{\beta} \right)V_n f \right\Vert_{\infty} \label{propo2inequality} .
\end{align}
To estimate the third term of \eqref{propo2inequality},  using the Nikol’skii’s inequality for trigonometric polynomials \cite{2006Nikol} and \eqref{hyperL2_unifromhelp},  we have
\begin{align*}
\left \Vert \left(\mathcal{U}_{0,L,N}- \mathcal{U}_{\lambda,L,N}^{\beta} \right)V_n f \right\Vert_{\infty} &\leq 2\sqrt{L}  \left \Vert \left(\mathcal{U}_{0,L,N}- \mathcal{U}_{\lambda,L,N}^{\beta} \right)V_n f \right\Vert_{L_2} \\
 &\leq \sqrt{\lambda L}  \Vert f\Vert_{W^{\phi,\beta}}.
\end{align*}
With the decomposition 
$\| f^\epsilon - V_n f \|_\infty \leq \| f^\epsilon - f \|_\infty + \| f - V_n f \|_\infty,$
and assuming that $c_1 = \inf \big\{t:  \|  \mathcal{U}_{\lambda,L,N}^\beta \|_\infty \geq t^{-1}  \big\},$ we obtain

\begin{align*}
\left \Vert \left(\mathcal{U}_{0,L,N}- \mathcal{U}_{\lambda,L,N}^{\beta} \right)V_n f \right\Vert_{\infty} & \leq 2\Vert  \mathcal{U}_{\lambda,L,N}^{\beta}\Vert_\infty \Vert \epsilon\Vert_\infty+2\Vert \epsilon\Vert_\infty +\sqrt{\lambda L}  \Vert f\Vert_{W^{\phi,\beta}}\\
&\leq (2c_1+2)\Vert  \mathcal{U}_{\lambda,L,N}^{\beta}\Vert_\infty \Vert \epsilon\Vert_\infty +\sqrt{\lambda L}  \Vert f\Vert_{W^{\phi,\beta}},
\end{align*}
and by setting $c=2c_1+2,$ which finally leads to \eqref{inferroranalysis}.
\end{proof}

\section{Parameter choice strategies}
In this section, we are interested in the choice of parameters for the trigonometric polynomial \eqref{soltutionp}, namely, the regularization parameter $\lambda$ and penalization parameters $\beta_{\ell,k}.$  To make sure the error bound \eqref{inferroranalysis} being finite, we fix the penalization parameter $\beta_{\ell,k}$ by two assumptions:
\begin{align*}
&1.\; \beta_{\ell,k} \;is \;nondecreasing,\;and\;\beta_{0,1}=0,\\
&2.\; \beta_{\ell,k}\; must\; follow \;that\; \Vert f\Vert_{W^{\phi,\beta}}=\sum_{\ell=0}^\infty \sum_{k=1}^{1+\gamma} \beta_{\ell,k}^2 |\left< f, Y_{\ell,k}\right>_{L_2}|^2<\infty.
\end{align*}
Obviously, to satisfy the second requirement,  different continuous functions $f$ will have different increasing rates of penalization parameter since their Fourier coefficients have different decreasing rates \cite[Exercises 18, Chapter 3]{stein2003fourier}. With certain penalization parameter, we could find a reasonable regularization parameter.

As the second term on the right-hand side of \eqref{inferroranalysis} increases with $\lambda$, while the upper bound of $\big\Vert \mathcal{U}_{\lambda,L,N}^\beta \big\Vert_\infty$ in the first term decreases, the parameter $\lambda$ must be chosen as a compromise between these two parts. We denote the optimal parameter that minimizes the error bound in \eqref{inferroranalysis} by $\lambda_{\mathrm{opt}}$.

Note the assumption that $L$ is large enough to ensure $\Vert V_n f - f \Vert_\infty \leq \Vert \epsilon \Vert_\infty$ in Section 4. Naturally, as $\Vert \epsilon \Vert_\infty \rightarrow 0$, one would let $\lambda_{\mathrm{opt}} \rightarrow 0$. Thus, the goal of finding a reasonable parameter $\lambda_{\mathrm{reg}}$ using some algorithm is to ensure that $\lambda_{\mathrm{reg}}$ inherits the asymptotic property of $\lambda_{\mathrm{opt}}$, i.e., $\lambda_{\mathrm{reg}} \rightarrow 0$ as $\Vert \epsilon \Vert_\infty \rightarrow 0$. At the same time, the designed $\lambda_{\mathrm{reg}}$ must remain related to the penalization parameter $\beta_{\ell,k}$. 

We summarize this idea with the following regularization definition.

\begin{definition}\label{regulardef}
 Let $\lambda_{{\rm reg}}(\boldsymbol{\beta}) $ be the parameter obtained by a parameter choice strategy related to the penalization parameter $\boldsymbol{\beta}$, and\; $\mathcal{U}_{\lambda_{{\rm reg}}(\boldsymbol{\beta} ),L,N}^\beta f^{\epsilon}$ be the noisy version of the approximation trigonometric polynomial \eqref{soltutionp}. A parameter choice strategy is said to be regular in the sense that 
\begin{align}\label{regulareual}
\Vert \mathcal{U}_{\lambda_{{\rm reg}}( \boldsymbol{\beta}),L,N}^\beta f^{\epsilon}-\mathcal{U}_{0,L,N} f^{\epsilon}  \Vert_\infty \leq C \Vert \epsilon \Vert_\infty,
\end{align}
where $C$ is a bounded constant independent of the noise $\epsilon.$
\end{definition}

\begin{remark}
Definition \ref{regulardef} indicates that $\lambda_{{\rm reg}} \to 0$ as $\Vert \epsilon \Vert_\infty \to 0$. This suggests that the strategy remains effective even at lower noise levels. To some extent, this behavior can be interpreted as a reflection of the stability of the parameter selection strategy. 
\end{remark}

In this paper, we adopt a heuristically motivated algorithm for parameter selection. The regularization parameters $\lambda_{\mathrm{reg}}$ are chosen from a finite set defined as  
\begin{align}\label{parameterset}
\mathcal{S} := \left\{ \lambda_k = \zeta_0 q^k \;\middle|\; k = 1, 2, \ldots, T \right\},
\end{align}
where $\zeta_0 > 0$, $q \in (0,1)$, and $T$ is sufficiently large.

\subsection{Laplace operator $\mathcal{R}_L$}
In this subsection, we derive a choice of the penalization operator $\mathcal{R}_L$ related to the Laplace operator $\Delta$ on $\mathbb{S}^1$, as it naturally satisfies the fixed assumptions for the penalization parameter $\beta_{\ell,k}$ outlined above.  For additional strategies on the selection of penalization parameters, we refer the reader to \cite{Parameterchoice2015}. The spherical harmonics on $\mathbb{S}^1$ have an intrinsic characterization as the eigenfunctions of the Laplace operator $\Delta,$ that is,
\begin{align*}
\Delta Y_{\ell,k}(x)=-\ell^2 Y_{\ell,k}(x).
\end{align*}
It follows that $-\Delta$ is a semipositive operator, and for any $s>0$ we may define $(-\Delta)^{s/2}$ by
\begin{align}\label{Laplaceoperator}
(-\Delta)^{s/2} Y_{\ell,k}(x)=\ell^{s} Y_{\ell,k}(x).
\end{align}
The corresponding matrix 
\begin{align*}
\mathbf{B}_L={\rm diag} \big{(} 0^{s}, 1^{s},1^{s},2^{s},2^{s}, \ldots,L^{s}, L^{s}  \big{)} \in \mathbb{R}^{(2L+1)\times (2L+1)}.
\end{align*}
In our next regularization parameter choice process, we fixed the penalization parameter $\beta_{\ell,k}$ by operator  $(-\Delta)^{s/2},$  namely
\begin{align*}
\beta_{\ell,1}=\beta_{\ell,2}=\ell^s,\;\;\ell=0,1,\ldots,L.
\end{align*}
 As we emphasized $\beta_{\ell,k}$ must make $\Vert f\Vert_{W^{\phi,\beta}}<\infty,$ thus we give extra assumption for $s$ that is
\begin{align*}
\sum_{\ell=0}^\infty \sum_{k=1}^{1+\gamma} \ell^{2s} |\left<f,Y_{\ell,k} \right>|^2 <\infty,\quad f\in C(\mathbb{S}^1 ).
\end{align*}

\subsection{Morozov's\; discrepancy\; principle}
As we have already emphasized the importance of the regularization parameter $\lambda>0$ for approximation quality, we will apply Morozov’s discrepancy principle, which is a posteriori choice to finish the task that determines the parameter $\lambda.$  The main idea of this method is aimed at designing an algorithm to find the unique parameter $\lambda^*$ that satisfies the following criterion
\begin{align}\label{criterion}
 \left\Vert \mathbf{W}^{1/2} \left(\mathbf{A}\boldsymbol{\alpha}_{\lambda^*}-\mathbf{f}^{\epsilon}\right) \right\Vert_2^2=\Vert \epsilon\Vert_2.
\end{align}
For more details of Morozov’s discrepancy principle, we refer the readers to \cite{Morozov1966OnTS}. We adopt criterion \eqref{criterion} as the parameter choice strategy of problem \eqref{matrixLeast}. Firstly, we introduce the $weighted\;2$-$norm$ for $\mathbb{R}^N$ as an auxiliary result, which is also adopted by Hesse and Le Gia \cite{L2errorestimates}
\begin{align}\label{auxiliary}
\Vert \mathbf{y}\Vert_{2,t_N} :=\left(\frac{2\pi}{N}\sum_{j=1}^N  y_j^2  \right)^{\frac{1}{2}},\;\;\mathbf{y}\in \mathbb{R}^N,
\end{align}
whose points $y_1,y_2,\ldots,y_N$ are the quadrature points of $N$-point trapezoidal rule.  Then we have to study the monotonicity of function $\left\Vert \mathbf{W}^{1/2} \left(\mathbf{A}\boldsymbol{\alpha}_\lambda-\mathbf{f}\right) \right\Vert_2^2$ about variable $\lambda$ for ensuring the uniqueness of parameter choice.
\begin{lemma}\label{increasing}
Under the conditions of Theorem \ref{theorem1}.   Define
\begin{align}
& J :\mathbb{R}^+\rightarrow \mathbb{R}^+,\;\;\;J(\lambda) :=\left\Vert \mathbf{W}^{1/2}\left( \mathbf{A}\boldsymbol{\alpha}_\lambda-\mathbf{f}\right) \right\Vert_2^2= \frac{2\pi}{N}\sum_{j=1}^N \left[ p_{\lambda,L,N}^\beta (x_j)-f (x_j)    \right]^2 \label{Jlambda}\\ \nonumber
\text{and}\\ 
& K:\mathbb{R}^+\rightarrow \mathbb{R}^+,\;\;\;K(\lambda) := \left\Vert \mathbf{W}^{1/2} \mathbf{A B}_L\boldsymbol{\alpha}_\lambda\right\Vert_2^2=\frac{2\pi}{N}\sum_{j=1}^N \left[ \left( (-\Delta)^{s/2} p_{\lambda,L,N}^\beta (x_j) \right) \right]^2, 
\end{align}  \\
where $ p_{\lambda,L,N}^\beta $ is the unique minimizer of problem \eqref{matrixLeast}, see \eqref{soltutionp}. Assume further that $\mathbf{A^T Wf}$ is a non-zero vector. Then 
\begin{flalign*}
&(i) K\; is\; continuous \;and \;strictly\; monotonic\; decreasing\; with\; \lambda.\\
&(ii)J\; is \;continuous\; and\; strictly\; monotonic\; increasing\;with\; \lambda.& 
\end{flalign*}
\end{lemma}
\begin{proof}
Since $p_{\lambda,L,N}^\beta \in \mathbb{P}_L,$ from \eqref{Laplaceoperator}, $ (-\Delta)^{s/2} p_{\lambda,L,N}^\beta \in \mathbb{P}_L,$ and the exactness of trapezoidal rule for $\left((-\Delta)^{s/2} p_{\lambda,L,N}^\beta \right)^2\in \mathbb{P}_{2L}$ yields       
\begin{align}\label{Klambda}
K(\lambda)=\sum_{j=1}^N \frac{2\pi}{N} \left[ \left((-\Delta)^{s/2} p_{\lambda,L,N}^\beta \right)(x_j) \right]^2=\int_{-\pi}^\pi \left[ \left((-\Delta)^{s/2} p_{\lambda,L,N}^\beta \right)(x) \right]^2 dx=\left\Vert (-\Delta)^{s/2} p_{\lambda,L,N}^\beta \right \Vert_{L_2}^2,
\end{align}
and computing $\left\Vert (-\Delta)^{s/2} p_{\lambda,L,N}^\beta \right \Vert_{L_2}^2$ in \eqref{Klambda}  with the help of Paseval's equality yields
\begin{equation} \label{Klambda1}
K(\lambda)=\sum_{\ell=0}^{L} \sum_{k=1}^{1+\gamma} \frac{\ell^{2s}}{(1+\lambda \ell^{2s})^2} \bigg{(}\frac{2\pi}{N}\sum_{j=1}^{N} f(x_j) Y_{\ell,k}(x_j)\bigg{)}^2,
\end{equation}
and the continuity of $K$ and monotonic immediately follows \eqref{Klambda1}.

We show that $J$ is strictly monotonic increasing with a proof by derivation, taking the first order condition of $J(\lambda)$ yields
\begin{align}
J'(\lambda)&=\frac{2\pi}{N} \sum_{j=1}^N  2\left[ \left(\sum_{\ell=0}^{L} \sum_{k=1}^{1+\gamma}  \frac{1}{1+\lambda \ell^{2s}} \left< f,Y_{\ell,k} \right>_N Y_{\ell,k}(x_j)\right)-f(x_j)   \right] \times \nonumber\\ &\qquad\qquad\qquad\qquad\qquad\qquad\qquad\left[\sum_{\ell=0}^L \sum_{k=1}^{1+\gamma} -\frac{\ell^{2s}}{(1+\lambda \ell^{2s})^2} \left<f,Y_{\ell,k} \right>_N Y_{\ell,k}(x_j)  \right]    \nonumber \\ 
&=J_1(\lambda)+J_2(\lambda), \label{J1J2}
\end{align}
where
\begin{align*}
 & J_1(\lambda)=-\frac{4\pi}{N}\sum_{j=1}^N   \left[ \sum_{\ell=0}^L \sum_{k=1}^{1+\gamma} \frac{1}{1+\lambda \ell^{2s}} \left<f,Y_{\ell,k} \right>_N Y_{\ell,k}(x_j)\right] \left[ \sum_{\ell=0}^L \sum_{k=1}^{1+\gamma} \frac{\ell^{2s}}{(1+\lambda \ell^{2s})^2} \left<f,Y_{\ell,k} \right>_N Y_{\ell,k}(x_j)\right]    ,\\
& J_2(\lambda)=\frac{4\pi}{N}\sum_{j=1}^N  f(x_j) \left[  \sum_{\ell=0}^L \sum_{k=1}^{1+\gamma} \frac{\ell^{2s}}{(1+\lambda \ell^{2s})^2} \left<f,Y_{\ell,k} \right>_N Y_{\ell,k}(x_j)\right] .
\end{align*}
Note that $ J_1(\lambda)$ can be rewritten by the discrete inner of two trigonometric polynomial $p_{\lambda,L,N}^\beta,\;t_{\lambda,L,N}^\beta \in \mathbb{P}_{2L},$ i.e., 
$$J_1(\lambda)=-2\left< p_{\lambda,L,N}^\beta, t_{\lambda,L,N}^\beta \right>_N,$$ 
where $t_{\lambda,L,N}^\beta(x)=\sum_{\ell=0}^{L} \sum_{k=1}^{1+\gamma} \frac{\ell^{2s}}{(1+\lambda \ell^{2s})^2} \left<f,Y_{\ell,k} \right>_N Y_{\ell,k}(x).$ Using the exactness of trapezoidal rule \eqref{exactness}, it is clear that
\begin{align}\label{J1}
J_1(\lambda)=-2\left< p_{\lambda,L,N}^\beta, t_{\lambda,L,N}^\beta \right>_N=-2\left< p_{\lambda,L,N}^\beta, t_{\lambda,L,N}^\beta \right>_{L_2}=-2\sum_{\ell=0}^{L} \sum_{k=1}^{1+\gamma} \frac{\ell^{2s}}{(1+\lambda \ell^{2s})^3} \left<f,Y_{\ell,k} \right>_N^2.
\end{align}
Substitute \eqref{J1} into \eqref{J1J2}, we have
\begin{align}
J'(\lambda)&=-2\sum_{\ell=0}^{L} \sum_{k=1}^{1+\gamma} \frac{\ell^{2s}}{(1+\lambda \ell^{2s})^3} \left<f,Y_{\ell,k} \right>_N+2\sum_{j=1}^N \frac{2\pi}{N} f(x_j) \left[ \sum_{\ell=0}^{L} \sum_{k=1}^{1+\gamma} \frac{\ell^{2s}}{(1+\lambda \ell^{2s})^2} \left<f,Y_{\ell,k} \right>_N Y_{\ell,k}(x_j)\right] \nonumber \\ 
&=-2\sum_{\ell=0}^{L} \sum_{k=1}^{1+\gamma} \frac{\ell^{2s}}{(1+\lambda \ell^{2s})^3} \left<f,Y_{\ell,k} \right>_N^2+2\sum_{\ell=0}^{L} \sum_{k=1}^{1+\gamma} \frac{\ell^{2s}}{(1+\lambda \ell^{2s})^2} \left<f,Y_{\ell,k} \right>_N \sum_{j=1}^N \frac{2\pi}{N} f(x_j) Y_{\ell,k}(x_j) \nonumber\\
&=-2\sum_{\ell=0}^{L} \sum_{k=1}^{1+\gamma} \frac{\ell^{2s}}{(1+\lambda \ell^{2s})^3} \left<f,Y_{\ell,k} \right>_N^2+2\sum_{\ell=0}^{L} \sum_{k=1}^{1+\gamma} \frac{\ell^{2s}}{(1+\lambda \ell^{2s})^2} \left<f,Y_{\ell,k} \right>_N^2\nonumber \\ 
&=\sum_{\ell=0}^{L} \sum_{k=1}^{1+\gamma} \frac{2\lambda \ell^{4s}}{(1+\lambda \ell^{2s})^3} \left<f,Y_{\ell,k} \right>_N^2. \label{Jdotlambda}
\end{align}
Since  we have already assumed that $\mathbf{A^TWf}\neq \boldsymbol{0},$  it follows that $J'(\lambda)>0,$ which shows the strictly monotonic increasing of $J(\lambda).$
\end{proof}

It is worth pointing out that the choice of the parameter $\lambda$ has to be made through a compromise between $J(\lambda)$ and $K(\lambda).$ Let $f^\epsilon\in \mathcal{C}( [-\pi,\pi])$, and $\mathbf{f^\epsilon}:=[f^\epsilon(x_1),f^\epsilon(x_2),\ldots, f^\epsilon(x_N)]^T.$ In the following theorem, we need to add more conditions to ensure the noise level $\Vert \epsilon\Vert_{2,t_N}$ could enter the range of $J(\lambda).$
 
\begin{theorem}\label{theordiscrenpency}
Under the conditions of Theorem \ref{theorem1}. Assume   \begin{align}\label{assumption}
\Vert \mathcal{L}_L f^\epsilon-f^\epsilon \Vert_{2,t_N}\leq \Vert \epsilon \Vert_{2,t_N} \leq \Vert f^\epsilon-\sigma \Vert_{2,t_N},
\end{align}
where $\sigma$ is the mean value of $\left\{f^\epsilon(x_j) \right\}_{j=1}^N,$ i.e., $\sigma=\left(\sum_{j=1}^N f^{\epsilon}(x_j)  \right)/N.$ Then there exists a unique $\lambda^*>0$ such that the unique solution $\boldsymbol{\alpha}_{\lambda^*}$ of \eqref{matrixLeast} satisfies
\begin{align*}
 \left\Vert \mathbf{W}^{1/2}\left( \mathbf{A}\boldsymbol{\alpha}_{\lambda^*}-\mathbf{f}^\epsilon \right) \right\Vert_2=\Vert \epsilon \Vert_{2,t_N}.
\end{align*}
\end{theorem}
\begin{proof}
We have to show that the function $F:(0,\infty)\rightarrow \mathbb{R}$ defined by
\begin{align}\label{discrepF}
F(\lambda)=\left\Vert  \mathbf{W}^{1/2}\left( \mathbf{A}\boldsymbol{\alpha}_{\lambda^*}-\mathbf{f}^\epsilon \right) \right\Vert_2^2 - \Vert \epsilon \Vert_{2,t_N} ^2
\end{align}
has a unique zero. From the representation \eqref{Jlambda}, we find that
\begin{align*}
F(\lambda)=J(\lambda)-\Vert \epsilon \Vert_{2,t_N} ^2=\frac{2\pi}{N}\sum_{j=1}^N \left[ \sum_{\ell=0}^{L} \sum_{k=1}^{1+\gamma} \frac{1}{1+\lambda \ell^{2s}} \left<f^\epsilon,Y_{\ell,k} \right>_N Y_{\ell,k}(x_j) -f^\epsilon(x_j)\right]^2-\Vert\epsilon \Vert_{2,t_N} ^2.
\end{align*}
Therefore, $F(\lambda)$ has following limits by the continuous of $J(\lambda)$
\begin{align*}
 \lim_{\lambda \rightarrow \infty}F(\lambda)&= \frac{2\pi}{N} \sum_{j=1}^N \bigg{(} \left<f^\epsilon, Y_{0,1} \right>_N Y_{0,1}(x_j)-f^\epsilon (x_j)   \bigg{)}^2-\Vert\epsilon \Vert_{2,t_N} ^2 \\
                                       &=\frac{2\pi}{N} \sum_{j=1}^N \left( \frac{ \sum_{k=1}^N  f^{\epsilon}(x_k)}{N}-f^\epsilon(x_j)     \right)^2-\Vert\epsilon \Vert_{2,t_N} ^2 \\
                                       &=\Vert f^\epsilon-\sigma\Vert_{2,t_N}^2-\Vert\epsilon \Vert_{2,t_N} ^2 \geq 0,
\end{align*}
and 
\begin{align*}
 \lim_{\lambda \rightarrow 0}F(\lambda)&= \frac{2\pi}{N} \sum_{j=1}^N \bigg{(} \sum_{\ell=0}^{L} \sum_{k=1}^{1+\gamma} \left<f^\epsilon, Y_{\ell,k} \right>_N Y_{\ell,k}(x_j)-f^\epsilon (x_j)   \bigg{)}^2-\Vert\epsilon \Vert_{2,t_N} ^2 \\
                                       &=\frac{2\pi}{N} \sum_{j=1}^N \big{(}\mathcal{L}_L f^\epsilon (x_j)-f^\epsilon(x_j)     \big{)}^2-\Vert\epsilon \Vert_{2,t_N} ^2 \\
                                       &=\Vert \mathcal{L}_L f^\epsilon-f^\epsilon \Vert_{2,t_N}-\Vert\epsilon \Vert_{2,t_N} ^2 \leq 0,
\end{align*}
while $J(\lambda)$ is strictly monotonically increasing, hence, $F$ has exactly one zero $\lambda^*.$ 
\end{proof}

As we state that the hyperinterpolation $\mathcal{L}_{L} f$ is also a trigonometric interpolation on an equidistant grid for $2L+1=N,$ which preserves the interpolation condition even in the presence of noise \cite[Remark 3.3]{AN2023115}, that is 
\begin{align*}
\mathcal{L}_L f^\epsilon(x_j)=f^\epsilon (x_j),\;j=1,2,\ldots, N.
\end{align*}
Thus, the lower bound of $\Vert \epsilon \Vert_{2,t_N}$ in assumption \eqref{assumption} would naturally equal zero. However, as shown in \cite{L2errorestimates}, it is difficult to estimate the specific lower bound of $\Vert \epsilon \Vert_{2,t_N}$ on the sphere, which is the main difference from the unit circle case. With this special case, Morozov’s discrepancy principle is regular.
\begin{corollary}\label{corollary1}
Let $\lambda_{{\rm mor}}$ be the root of the discrepancy function $F(\lambda)$ and $2L+1=N.$ The Morozov’s discrepancy principle is regular in the sense that if $\Vert \epsilon \Vert_\infty \rightarrow 0,$ then
\begin{align}\label{regular}
\Vert \mathcal{U}_{\lambda_{{\rm mor}},L,N}^\beta f^{\epsilon}-\mathcal{U}_{0,L,N} f^{\epsilon} \Vert_\infty \leq \sqrt{(2L+1)2\pi}\Vert \epsilon\Vert_{\infty}.
\end{align}
\end{corollary}
\begin{proof}
By the expression \eqref{soltutionp} and \eqref{triginter}, it follows that 
\begin{align*}
\Vert \mathcal{U}_{\lambda,L,N}^\beta f^{\epsilon}-\mathcal{U}_{0,L,N} f^{\epsilon}  \Vert_\infty&=\max_{x\in [-\pi,\pi]} \left| \sum_{\ell=0}^{L} \sum_{k=1}^{1+\gamma} \frac{\lambda_{{\rm mor}} \ell^{2s}}{ 1+\lambda_{{\rm mor}} \ell^{2s}} \left<f^\epsilon, Y_{\ell,k} \right>_N Y_{\ell,k}(x)\right|, \\
&\leq   \sum_{\ell=0}^{L} \sum_{k=1}^{1+\gamma} \left|\frac{\lambda_{{\rm mor}} \ell^{2s}}{ 1+\lambda_{{\rm mor}} \ell^{2s}} \left<f^\epsilon, Y_{\ell,k} \right>_N \right| \\
&\leq \left( 2L+1 \right)^{1/2}\left(\sum_{\ell=0}^{L} \sum_{k=1}^{1+\gamma} \left(\frac{\lambda_{{\rm mor}} \ell^{2s}}{ 1+\lambda_{{\rm mor}} \ell^{2s}} \left<f^\epsilon, Y_{\ell,k} \right>_N \right)^2\right)^{1/2},
\end{align*}
and from the Morozov’s discrepancy principle, we have 
\begin{align*}
\Vert\epsilon \Vert_{2,t_N} ^2&=\left\Vert  \mathbf{W}^{1/2}\left( \mathbf{A}\boldsymbol{\alpha}_{\lambda_{{\rm mor}}}-\mathbf{f}^\epsilon \right) \right\Vert_2^2 \\
&=\frac{2\pi}{N} \sum_{j=1}^N \bigg{(} \sum_{\ell=0}^{L} \sum_{k=1}^{1+\gamma} \frac{1}{1+\lambda_{{\rm mor}} \ell^{2s}} \left<f^\epsilon, Y_{\ell,k} \right>_N Y_{\ell,k}(x_j)-f^\epsilon (x_j)   \bigg{)}^2 \\
&=\frac{2\pi}{N} \sum_{j=1}^N \bigg{(} \sum_{\ell=0}^{L} \sum_{k=1}^{1+\gamma} \frac{1}{1+\lambda_{{\rm mor}} \ell^{2s}} \left<f^\epsilon, Y_{\ell,k} \right>_N Y_{\ell,k}(x_j)-\sum_{\ell=0}^{L} \sum_{k=1}^{1+\gamma}  \left<f^\epsilon, Y_{\ell,k} \right>_N Y_{\ell,k}(x_j)   \bigg{)}^2 \\
&=\frac{2\pi}{N} \sum_{j=1}^N \bigg{(} \sum_{\ell=0}^{L} \sum_{k=1}^{1+\gamma} \frac{\lambda_{{\rm mor}} \ell^{2s}}{1+\lambda_{{\rm mor}} \ell^{2s}} \left<f^\epsilon, Y_{\ell,k} \right>_N Y_{\ell,k}(x_j)   \bigg{)}^2 \\
&= \sum_{\ell=0}^{L} \sum_{k=1}^{1+\gamma} \left(\frac{\lambda_{{\rm mor}} \ell^{2s}}{1+\lambda_{{\rm mor}} \ell^{2s}} \left<f^\epsilon, Y_{\ell,k} \right>_N\right)^2. 
\end{align*}
The second equality is from the interpolation condition \eqref{intercondition}, and the last equality is from the exactness of the trapezoidal rule \eqref{exactness}. By the auxiliary result \eqref{auxiliary}, we have $ \Vert \epsilon \Vert_{2,t_N}^2 \leq 2\pi \Vert \epsilon\Vert_{\infty}^2,$ and thus
\begin{align*}
\Vert \mathcal{U}_{\lambda_{{\rm mor}},L,N}^\beta f^{\epsilon}-\mathcal{U}_{0,L,N} f^{\epsilon} \Vert_\infty \leq \sqrt{(2L+1)2\pi}\Vert \epsilon\Vert_{\infty},
\end{align*}
which shows \eqref{regular}.
\end{proof}

In practice,  we do not need to determine $\lambda_{{\rm mor}}$ satisfying $F(\lambda_{{\rm mor}})=0$ exactly. Usually we can also choose moderately sized $\lambda_{{\rm mor}}$ from parameter set \eqref{parameterset} by stopping criterion $F(\lambda_{{\rm mor}})<0.$

\subsection{L-curve}
Although Morozov’s discrepancy principle is an efficient parameter selection strategy, from Theorem \ref{theordiscrenpency} we know that it requires prior knowledge of the noise level $\Vert \epsilon \Vert_{2,t_N}.$ In fact, the noise level is often unknown, so alternative methods that do not rely on noise information are needed for parameter selection.

A famous regularization parameter choice strategy that does not require noise knowledge is the L curve \cite{Hansen2007,Hansen1993}. Considering our least squares model \eqref{matrixLeast}, the L-curve is a parametric log-log plot for two parts: $\hat{\rho}(\lambda)=\log J(\lambda)$ and $\hat{\eta}(\lambda)=\log K(\lambda).$ Surprisingly, the resulting curve $(\hat{\rho}(\lambda), \hat{\eta}(\lambda))$ is L-shaped, and the optimal parameter is chosen corresponding to the corner.  We usually determine the corner by the maximal curvature $\kappa(\lambda),\;\lambda>0$ of this curve, which is defined as follows
\begin{align}\label{curvatureformula}
 \kappa(\lambda):=\frac{\left| \hat {\rho}'(\lambda)\hat{\eta}''(\lambda)-\hat{\eta}'(\lambda) \hat{\rho}''(\lambda)  \right|    }{ \left(\hat{\rho}'^{2}(\lambda)+\hat{\eta}'^2(\lambda)  \right)^{3/2 }           }.
\end{align}
It is very important to find the corner of the L-curve using a practical algorithm. There are several algorithms to find this corner \cite{Cultrera_2020,An_algorithm_Lcurve}, we adopt the most direct method by computing its curvature \cite{Lcurve2001}.  In the following theorem, we give the maximal curvature of the curve $(\log J(\lambda), \log K(\lambda)).$

The following theorem is similar to \cite[Section 5]{Lcurve2001}. However, our proof method is based on the exactness of the trapezoidal rule, rather than the one in 
\cite{Lcurve2001} is based on the singular value decomposition.
\begin{theorem}\label{theoremcurvature}
Adopt the conditions of Lemma \ref{increasing}.   Then the curvature $\kappa(\lambda)$ of curve $(\log \rho,\log  \eta)$ is given by
\begin{align}\label{curvatureformula1}
\kappa(\lambda)=\frac{\rho \eta}{\eta'} \frac{  \lambda \eta' \rho+\rho \eta+\lambda^2 \eta'^2\eta}{ (\lambda^2 \eta^2+\rho^2)^{3/2}},
\end{align}
where $\rho=J(\lambda),\;\eta=K(\lambda),$ respectively.
\end{theorem}
\begin{proof}
As we already finish the computation of $J'(\lambda)$(see \eqref{Jdotlambda}), it follows that
\begin{align*}
\rho'=\sum_{\ell=0}^L \sum_{k=1}^{1+\gamma} \frac{2\lambda \ell^{4s}}{(1+\lambda \ell^{2s})^3} \left<f,Y_{\ell,k} \right>_N^2.
\end{align*}
From the representation of $K(\lambda)$ \eqref{Klambda1},  then
\begin{align*}
\eta'=-\sum_{\ell=0}^L \sum_{k=1}^{1+\gamma} \frac{2 \ell^{4s}}{(1+\lambda \ell^{2s})^3} \left<f,Y_{\ell,k} \right>_N^2.
\end{align*}
Note that $\rho'=-\lambda \eta',$ thus
\begin{align*}
\rho''=-\lambda \eta''-\eta'.
\end{align*}
Since $(\log \rho)'=\frac{\rho'}{\rho},\;(\log \eta)'=\frac{\eta'}{\eta},$  we have curvature by insert $\rho,\eta,\eta',\eta''$ into the curvature formula \eqref{curvatureformula}, we then obtain \eqref{curvatureformula1}.
\end{proof}

Obviously, the L-curve plot does not depend on any knowledge of noise level, i.e., $\Vert \epsilon\Vert_{2,t_N}.$  However, various results that do not satisfy the regular definition \eqref{regulareual} have been established for the L-curve criterion\cite{1996Limitations,1992Analysis,Vogel1996Non}.

\subsection{Generalized\;Cross-Validation}
For the least squares problem  \eqref{matrixLeast} with unknown noise level, one can adopt a statistical method called generalized cross-validation (GCV) to obtain a proper parameter $\lambda.$ The GCV estimate of $\lambda$ is the minimizer of $V(\lambda)$ given by
 \begin{align}\label{gcvfuntion}
V(\lambda)=\frac{\left\Vert  \mathbf{W}^{1/2} \left(\mathbf{A}\boldsymbol{\alpha}_{\lambda}-\mathbf{f}^\epsilon \right) \right\Vert_2^2}{\left [{\rm Tr}(\mathbf{I-A(\lambda)})\right]^2},
\end{align}
where $\mathbf{A}(\lambda)=\mathbf{W}^{1/2}\mathbf{A}(\mathbf{A^T \mathbf{W}A}+\lambda\boldsymbol{\beta})^{-1}\mathbf{A^T}\mathbf{W}^{1/2},$  and ${\rm Tr}(\cdot)$ denotes the trace of a matrix. In this paper, we do not give the motivation for the GCV function \eqref{gcvfuntion}; for more details, we refer to \cite{GenH1979}. 

Although generalized cross-validation is a widely used tool for determining a regularization parameter, challenges arise when evaluating the trace of an inverse matrix in $V(\lambda)$ for large-scale problems. Since  $\mathbf{A^T\mathbf{W} A}$ is an identity matrix (see \eqref{identity}), we can compute the trace in $V(\lambda)$ directly, thereby simplifying the calculation.

\begin{theorem}
For the regularized least squares problem \eqref{matrixLeast}. If $N=2L+1,$ then the generalized cross-validation function $V(\lambda)$ can be expressed as follows
\begin{align}\label{Sgcvfunction}
V(\lambda)&=\frac{\sum_{\ell=0}^L \sum_{k=1}^{1+\gamma}\left(\frac{\lambda \ell^{2s}}{1+\lambda\ell^{2s}}\left<f^\epsilon,Y_{\ell,k} \right>_N  \right)^2}{\left[\sum_{\ell=1}^L \frac{2\lambda \ell^{2s}}{(1+\lambda \ell^{2s} )}\right]^2}.
\end{align}
\end{theorem}
\begin{proof}
By direct computing, we have
\begin{align*}
{\rm Tr}(\mathbf{I-A(\lambda)})&={\rm Tr}(\mathbf{I}-\mathbf{W}^{1/2}\mathbf{A}( \mathbf{A^T \mathbf{W}A+\lambda\boldsymbol{\beta})^{-1}A^T}\mathbf{W}^{1/2})\\
                              &={\rm Tr}(\mathbf{I})- {\rm Tr}(\mathbf{A^T W A(I+\lambda\boldsymbol{\beta})^{-1}}) \nonumber\\
                              &=2L+1-\sum_{\ell=0}^L \sum_{k=1}^{1+\gamma}\frac{1}{(1+\lambda \ell^{2s} )} \nonumber \\&=\sum_{\ell=1}^L \frac{2\lambda \ell^{2s}}{(1+\lambda \ell^{2s} )}. 
\end{align*}
Note that the equivalent form of $\left\Vert  \mathbf{W}^{1/2}(\mathbf{A}\boldsymbol{\alpha}_{\lambda}-\mathbf{f}^\epsilon )\right\Vert_2^2$ in the proof of Corollary \ref{corollary1}, and with above trace, we have \eqref{Sgcvfunction}. 
\end{proof}

From \eqref{Sgcvfunction}, we could also give the lower and upper bounds for  $V(\lambda).$

\begin{corollary}
Let $z_{\min}$ and $z_{\max}$ be the maximum value and minimum value of $\left< f^{\epsilon},Y_{\ell,k} \right>_N,$ respectively. Then 
\begin{align}\label{estimation_gcv}
\frac{1}{2} \left(\frac{\lambda z_{\min}}{1+\lambda} \right)^2     \leq V(\lambda) \leq \frac{(1+\lambda)^2 z_{\max}^2}{2\lambda^2}.
\end{align} 
\end{corollary}
\begin{proof}
For the bounds of the numerator in \eqref{Sgcvfunction}, we have 
\begin{align*}
2L\left(\frac{\lambda z_{\min}}{1+\lambda}\right)^2\leq \sum_{\ell=0}^L \sum_{k=1}^{1+\gamma}\left(\frac{\lambda \ell^{2s}}{1+\lambda\ell^{2s}}\left<f^\epsilon,Y_{\ell,k} \right>_N  \right)^2\leq 2Lz_{\max}^2.
\end{align*}
Similarly, we have the estimation of the denominator in \eqref{Sgcvfunction},
\begin{align*}
4L \left(\frac{\lambda}{1+\lambda}  \right)^2 \leq \left[\sum_{\ell=1}^L \frac{2\lambda \ell^{2s}}{(1+\lambda \ell^{2s} )}\right]^2 \leq 4L,
\end{align*}
thus, we could obtain the estimation \eqref{estimation_gcv}.
\end{proof}
\begin{remark}
In \cite{GCVforlarge1999},  it can be seen that there is no exact value for the GCV function with different parameter $\lambda,$ since  
this function involves the solution of linear systems (see the inverse matrix in \eqref{gcvfuntion}). In that case, the authors adopt several iterative algorithms to estimate GCV function. However, thanks to the special properties of trapezoidal rule nodes, we can compute the new GCV function $V(\lambda)$ \eqref{Sgcvfunction} directly, which could reduce the computation costs. 
\end{remark}

 From our above analyses, similar to L-curveother parameter choice strategies do, GCV also does not require the norm of noise level $\epsilon.$  Of course, there are other parameter choice strategies that do not require noise level information, e.g.,  quasi-optimality criterion \cite{TIKHONOV196593}. Unfortunately, by a theorem of Bakushinskii \cite{BAKUSHINSKII1984181}, the regular result \eqref{regulareual} cannot hold for any parameter selection method that is noise level free, meaning that such methods, including the L-curve, are also inherently unstable.

\begin{remark}
We have to point out that Bakushinskii's result does not imply that the L-curve and GCV will always fail to \eqref{regulareual}. Thus, we verify the efficiency and stability of L-curve and GCV for our regularization problem by numerical experiments instead of theoretical analysis. 
\end{remark}

\subsection{Algorithms}

According to the above results, we can design corresponding algorithms for these three parameter choice strategies. We set the Laplace operator as our regularization operator $\mathcal{R}_L.$ For comparison, we first find the optimal parameter $\lambda_{{\rm opt}}$ from the finite set \eqref{parameterset}. Note that we always fix $L=(N-1)/2$ in our algorithms, as the degree is not regarded as a parameter. Let $x^*$ denote the discrete sample vector in $[-\pi, \pi]$, and let $f^\epsilon$ be its corresponding noisy version.

For Morozov’s discrepancy principle, this method requires us to find $\lambda$ such that the value $J(\lambda)$ defined by \eqref{Jlambda} satisfies $J(\lambda)=\Vert \epsilon\Vert_{2,t_N}^2$. When the relation between the number of points and the degree of approximation trigonometric polynomial \eqref{soltutionp} satisfies $2L+1=N$, this method is regular from Corollary \ref{corollary1}. To find the zero point of $J(\lambda)=\Vert \epsilon\Vert_{2,t_N}^2$, we just need to compute the value of $J(\lambda_k)$ until $F(\lambda_k)= J(\lambda_k)-\Vert \epsilon\Vert_{2,t_N}^2$ becomes negative.

The algorithm for finding $\lambda_{{\rm mor}}$ is listed in Algorithm \ref{alg:morozov}.

\begin{algorithm}[H]
\ Don't Print Semicolon
  \SetAlgoLined
 \KwIn {$x^*, \;\mathbf{f^{\epsilon}},$ points number $N,$ penalization parameter $\boldsymbol{\beta},$ the cardinal number $T$ of parameter set $\mathcal{S}$, noise level $\Vert \epsilon\Vert_{2,t_N}$;}
Stopping criterion : $\Vert \mathbf{W}^{1/2}\left(\mathbf{A}\alpha_k-\mathbf{f^\epsilon}\right)\Vert_{2,t_N}^2-\Vert \epsilon \Vert_{2,t_N}^2>0$ or $k>T;$\;
Initialization : Degree $L=(N-1)/2,\;k=0;$\;
  \While {$F_k < 0 \;\&\; k<= T$}{
    $\alpha_k=\frac{1}{1+\lambda_k \boldsymbol{\beta}.^2} \mathbf{A^T W f^\epsilon};$ \;
    $F_k=\Vert \mathbf{W}^{1/2}\left( \mathbf{A}\alpha_k-\mathbf{f^\epsilon} \right)\Vert_{2,t_N}^2-\Vert \epsilon \Vert_{2,t_N}^2;$ 
     $k\gets k+1;$\;
  }
  \KwOut{$\lambda_{{\rm mor}}=\lambda_k.$}
  \caption{Calculate parameter $\lambda_{{\rm mor}}$ from $\mathcal{S}$}
  \label{alg:morozov}
\end{algorithm}

Another strategy is the L-curve method, which recommends finding the maximal curvature of the log-log plot: $( J(\lambda),K(\lambda))$. This plot is L-shaped, and the maximal curvature is usually called the ``corner'' of this curve. To find this corner, we compute the maximal value of the curvature formula $\kappa(\lambda)$ defined by \eqref{curvatureformula1}. This method does not require concrete noise level information, and we do not provide further assumptions for it, which makes it potentially fail for some models. Also, since we do not consider any theoretical analysis of this method, we find the maximum value of \eqref{curvatureformula1} by comparing all parameters in the finite set $\mathcal{S}$.

The algorithm for finding $\lambda_{{\rm corner}}$ is listed in Algorithm \ref{alg:lcurve}.

\begin{algorithm}[H]
\DontPrintSemicolon
  \SetAlgoLined
 \KwIn {$x^*, \;\mathbf{f^{\epsilon}},$ points number $N,$  penalization parameter $\boldsymbol{\beta},$ the cardinal number $T$ of parameter set $\mathcal{S}.$}
Stopping criterion : $k>T;$\;
Initialization : Degree $L=(N-1)/2,\;k=0;$\;
  \If {$k < =T$}{
    $\alpha_k=\frac{1}{1+\lambda_k \boldsymbol{\beta}.^2}\mathbf{ A^T W f^\epsilon};$ \;
    $\rho_k=\log(\Vert \mathbf{A} \alpha_k-\mathbf{f^\epsilon}\Vert_{2,t_N})$ \;
    $\eta_k=\log(\Vert \mathbf{A} (\boldsymbol{\beta}\alpha_k)\Vert_{2,t_N})$ \;
    Computing $\kappa_k$ by \eqref{curvatureformula1}\;
     $k\gets k+1;$\;
  }
  \KwOut{$\lambda_{{\rm corner}}={\rm arg} \min_{\lambda_k} \kappa_k.$}
  \caption{Calculate corner of L-curve $\lambda_{{\rm corner}}$ from $\mathcal{S}$}
  \label{alg:lcurve}
\end{algorithm}

Without noise information, the GCV estimate is popular for the selection of parameter $\lambda$, which is also just the minimizer of the GCV function $V(\lambda)$. However, we use the same searching method as L-curve due to the lack of theoretical support.

The algorithm for finding $\lambda_{{\rm gcv}}$ is listed in Algorithm \ref{alg:gcv}.

\begin{algorithm}[H]
\DontPrintSemicolon
  \SetAlgoLined
 \KwIn {$x^*, \;\mathbf{f^{\epsilon}},$ points number $N,$ penalization parameter $\boldsymbol{\beta},$ the cardinal number $T$ of parameter set $\mathcal{S};$}
Stopping criterion : $k>T;$\;
Initialization : Degree $L=(N-1)/2,\;k=0;$\;
  \If {$k <= T$}{
    $\alpha_k=\frac{\lambda_k \boldsymbol{\beta}.^2}{1+\lambda_k \boldsymbol{\beta}.^2} \mathbf{A^T W f^\epsilon};$ \;
    Computing ${\rm trace}_k$ by \eqref{Sgcvfunction}\;
    $ v_k=\Vert \alpha_k \Vert_2^2/{\rm trace}_k^2$ \;
     $k\gets k+1;$\;
  }
  \KwOut{$\lambda_{{\rm gcv}}={\rm arg} \min_{\lambda_k} v_k.$}
  \caption{Calculate the minimizer of GCV function $\lambda_{{\rm gcv}}$ from $\mathcal{S}$}
  \label{alg:gcv}
\end{algorithm}

\section{Numerical experiments}

In this section, we report numerical results to illustrate the theoretical results derived above and test the approximation quality of \eqref{soltutionp}. There are two testing functions as follows: \\

periodic entire function \cite{ExpTrapezoidal}
\begin{align*}
f_1(x)=\exp\left(\cos x \right),
\end{align*}

and a periodic entire function with high-frequency oscillation
\begin{align*}
f_2(x)=\exp\left(\cos x\right)+\sin 30 x.
\end{align*}

The level of noise is measured by $signal$-$to$-$noise$\; $radio \;(SNR),$  which is defined as the ratio of signal to the noisy data, and is often expressed in decibels (dB). For given clean signal $\mathbf{f}\in \mathbb{R}^{N\times 1},$ we add noise to this data
\begin{align*}
\bf{d}=\bf{f}+\alpha \bf{\epsilon},
\end{align*}                                                 
where $\alpha$ is a scalar used to yield a predefined SNR, $\bf{\epsilon}$ is a vector following a Gaussian distribution with mean value 0. Then we give the definition of SNR:
\begin{align*}
{\rm SNR}:=10 \log_{10} \left( \frac{P_{{\rm signal}}}{\alpha P_{{\rm noise}}}      \right),
\end{align*}
where $P_{{\rm signal}}=\sqrt{\frac{1}{N}\sum_{k=1}^N f^2(x_k)},\; P_{\rm{noise}}$ is the standard deviation of $\bf{\epsilon}.$  A lower scale of  SNR suggests more noisy data. To test the approximation quality, we use an equidistant point set $\mathcal{X}\subset\mathbb{S}^1$ to (approximately) determine the $L_2$ error and uniform error, which is estimated as follows:
\begin{align*}
& L_2 \;{\rm error} \approx \left(\frac{2\pi}{N} \sum_{j=1}^K ( p_{\lambda, L,N}^\beta (x_j)-f(x_j) )^2 \right)^{1/2},\;\;x_j \in \mathcal{X},\\
& {\rm Uniform}\;{\rm error} \approx \max_{j=1,\ldots,K} \left| p_{\lambda, L,N}^\beta (x_j)-f(x_j) \right|, \;\;\;x_j \in \mathcal{X}.
\end{align*}

In all our experiments, we assume that the set of points consists of equidistant points on the unit circle \(\mathbb{S}^1\). When selecting parameters using the three algorithms outlined in Section 5, we ensure that the relationship between the number of points \(N\) and the degree of the trigonometric polynomial approximation \(L\) satisfies \(2L + 1 = N\), as shown in \eqref{soltutionp}. This choice is made for two reasons: first, to fix the degree \(L\) for more convenient comparisons between these algorithms, and second, because Morozov's discrepancy principle is regular under the condition \(2L + 1 = N\), as proven in Corollary \ref{corollary1}.

For each value of \(\lambda\), we compute the \(L_2\) approximation error and uniform error using \(p_{\lambda,L,N}^\beta\) from \eqref{soltutionp} for the functions \(f_1(x)\) and \(f_2(x)\). The error curves are displayed in Figure \ref{opt_eps}, where it is evident that an appropriate choice of parameters significantly enhances the quality of the approximation.
\begin{table}[ht]
\centering
\caption{$\lambda_{{\rm opt}},\;\lambda_{{\rm corner}},\;\lambda_{{\rm mor}},\;\lambda_{{\rm gcv}}$ for $f_1$ with different noise level}
\begin{tabular}{ccccc}
\toprule
Error Level (dB) & $\lambda_{{\rm opt}}$    &  $\lambda_{{\rm corner}}$  &  $\lambda_{{\rm mor}}$  & $\lambda_{{\rm gcv}}$   \\
\midrule
10          & 0.0078125  & 0.10882    & 0.03125    & 8.6317e-05 \\
20          & 0.0024046  & 0.0051543  & 0.0055243  & 0.0022436  \\
30          & 0.00048828 & 0.00091117 & 0.0020933  & 0.0007401  \\
40          & 0.00026166 & 0.00016107 & 0.0007401  & 0.00016107 \\
50          & 7.0111e-05 & 2.4788e-05 & 0.00024414 & 6.5416e-05 \\
60          & 2.8474e-05 & 2.5168e-06 & 8.0536e-05 & 2.3128e-05 \\
70          & 1.3284e-05 & 2.7387e-07 & 2.8474e-05 & 1.0067e-05 \\
80          & 4.0885e-06 & 2.9802e-08 & 1.2394e-05 & 3.5592e-06 \\
\bottomrule
\end{tabular}
\end{table}

\begin{table}[ht]
\centering
\caption{$\lambda_{{\rm opt}},\;\lambda_{{\rm corner}},\;\lambda_{{\rm mor}},\;\lambda_{{\rm gcv}}$ for $f_2$ with different noise level}
\begin{tabular}{ccccc}
\toprule
Error Level (dB) & $\lambda_{{\rm opt}}$    &  $\lambda_{{\rm corner}}$  &  $\lambda_{{\rm mor}}$  & $\lambda_{{\rm gcv}}$   \\
\midrule
10          & 1.6859e-07  & 1.6859e-07 & 5.8706e-07  & 1.4676e-07 \\
20          & 5.5613e-08  & 4.2147e-08 & 1.6859e-07  & 5.1889e-08 \\
30          & 1.4901e-08  & 7.9853e-09 & 5.9605e-08  & 1.5971e-08 \\
40          & 6.0518e-09  & 0.125      & 2.1073e-08  & 4.9156e-09 \\
50          & 1.9963e-09  & 0.125      & 7.4506e-09  & 1.5129e-09 \\
60          & 7.5647e-10  & 0.125      & 2.6342e-09  & 2.4954e-10 \\
70          & 2.6745e-10  & 0.125      & 8.6895e-10  & 9.0949e-13 \\
80          & 4.4113e-11  & 0.125      & 2.8665e-10  & 9.0949e-13 \\
\bottomrule
\end{tabular}
\end{table}

Then we present the complete log-log plot of $\left(J(\mathcal{S}), K(\mathcal{S})\right)$ for $f_1(x),\;f_2(x)$ with 20 dB noise.  Figure \ref{L_curve_eps} illustrates that both plots exhibit an L-shaped curve, with the L-corner of $f_1(x)$ being more prominent than that of $f_2(x).$

To test the denoising ability of these three parameter choice strategies for different noise levels, we employed the same settings as in the previous experiments and added a decreasing sequence of noise levels from 10 dB to 100 dB with a step of 5 dB to $f_1(x)$ and $f_2(x).$ We then plotted the corresponding parameter values $\lambda_{\rm opt}$, $\lambda_{\rm mor}$, $\lambda_{\rm corner}$, and $\lambda_{\rm gcv}$ on a semi-logarithmic scale.  Table 1 and Table 2 indicate that Morozov's discrepancy principle is regular for both functions, which verifies the conclusion of Corollary \ref{corollary1}. However, the L-curve method and GCV can only work for relatively higher noise. In practice, we should avoid using the L-curve method for cases of low noise levels; instead, Morozov's discrepancy principle is stable under those conditions. In summary, the choice among these three strategies should depend on the approximate noise level.

Finally, we compare the recovery efficiency of continuous periodic functions using the approximation trigonometric polynomial $p_{\lambda,L,N}^\beta$ with parameters $\lambda_{\rm gcv}$ and $\lambda_{\rm corner}$, against $p_{\lambda,L,N}^\beta$ with a randomized parameter $\lambda_{\rm random}$, which follows a normalized distribution on $(0, 0.1)$. Note that we do not compare Morozov's discrepancy principle, as it requires verifying assumption \eqref{assumption} and noise level information, making it primarily suitable for theoretical analysis. To achieve this goal, we selected more general examples from CHEBFUN 5.7.0 \cite{trefethon2017chebfun}. In this tool, the command \texttt{cheb.gallerytrig('name')} provides several classical periodic functions, and we chose \texttt{'tsunami'} as our testing example. For more details about gallery functions, refer to \cite{ExtPeriodic2015}. We added Gaussian white noise at 10 dB to the above function. Figure \ref{recover_eps} shows that the approximation scheme \eqref{soltutionp} with parameters $\lambda_{\rm gcv}$ and $\lambda_{\rm corner}$ effectively recovers general continuous periodic functions, outperforming the randomized parameter selection.

\begin{figure}[htbp]
  \centering
  \includegraphics[width=1.1\textwidth]{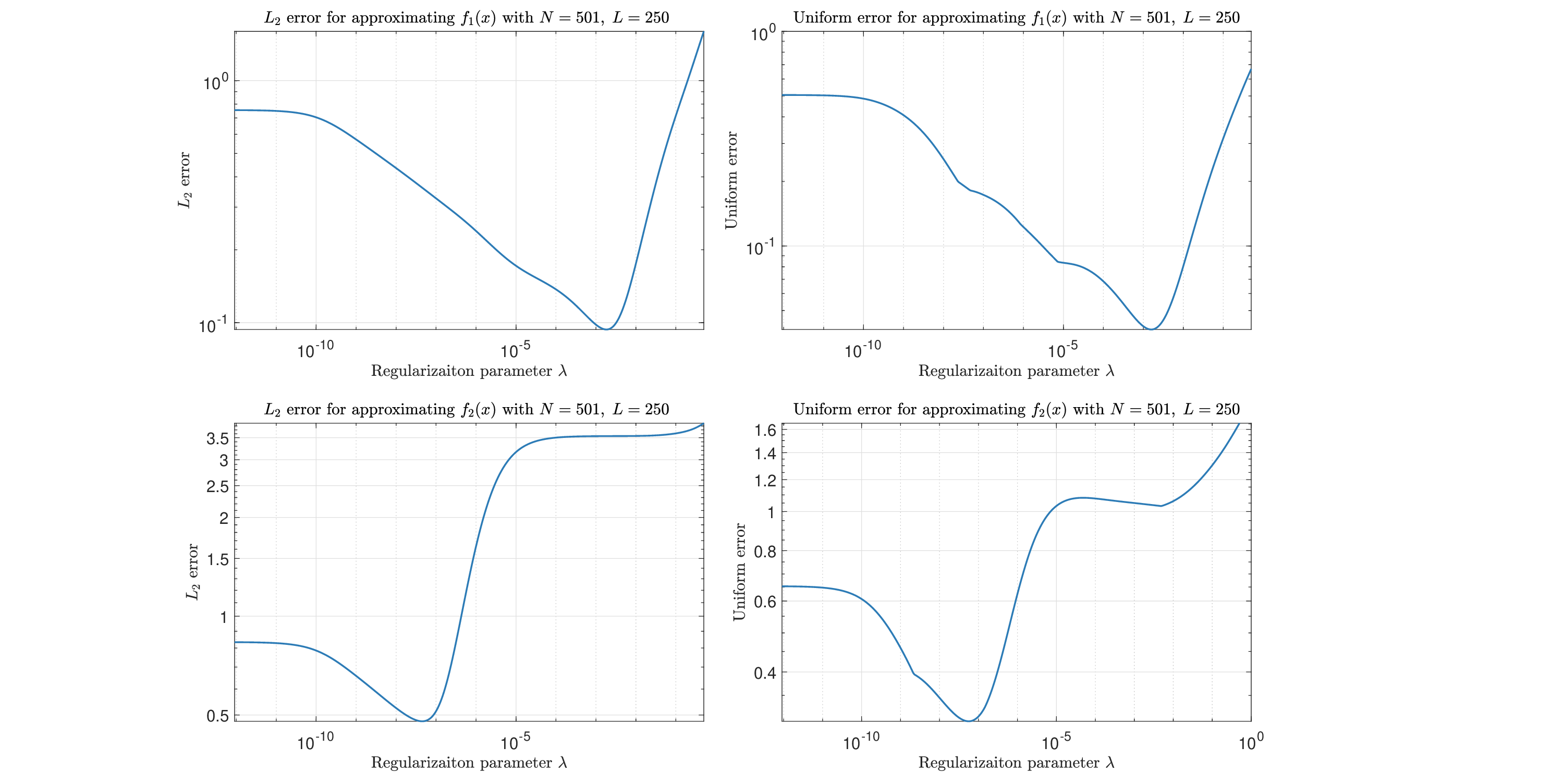}\\
\caption{$L_2$ errors (left) and uniform errors (right) as a function of $\lambda$ for the trigonometric polynomial \eqref{soltutionp} approximation to $f_1$ and $f_2$ with $N=501,\;L=250$ and 20 dB noise.} \label{opt_eps}
\end{figure}

\begin{figure}[htbp]
  \centering
  \includegraphics[width=\textwidth]{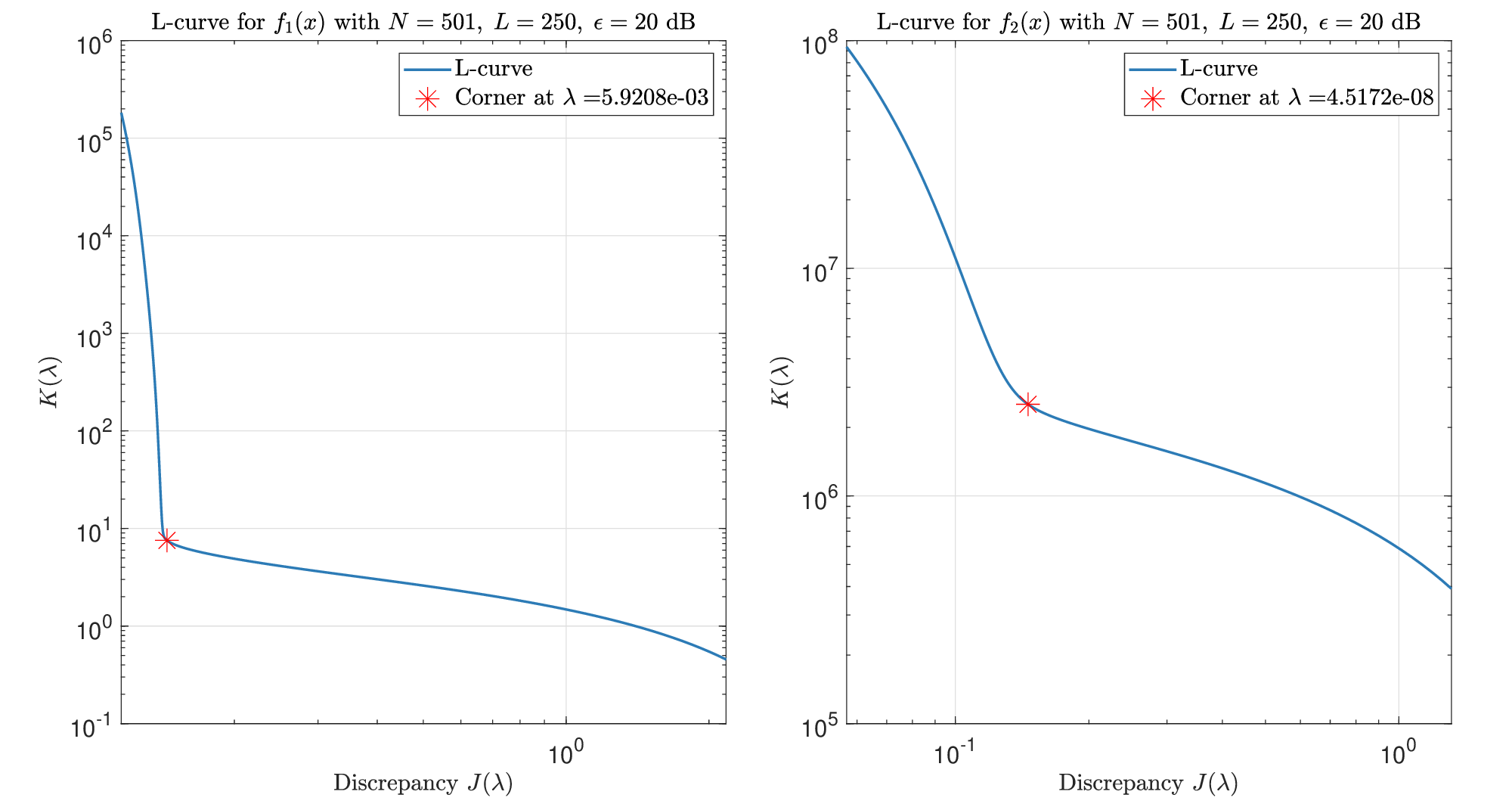}\\
\caption{L-curve method: log-log plot of $K(\lambda)$ against $J(\lambda)$ for $f_1,\;f_2$ with $N=501,\;L=250$ and 20 dB noise.} \label{L_curve_eps}
\end{figure}

\begin{figure}[htbp]
  \centering
  \includegraphics[width=\textwidth]{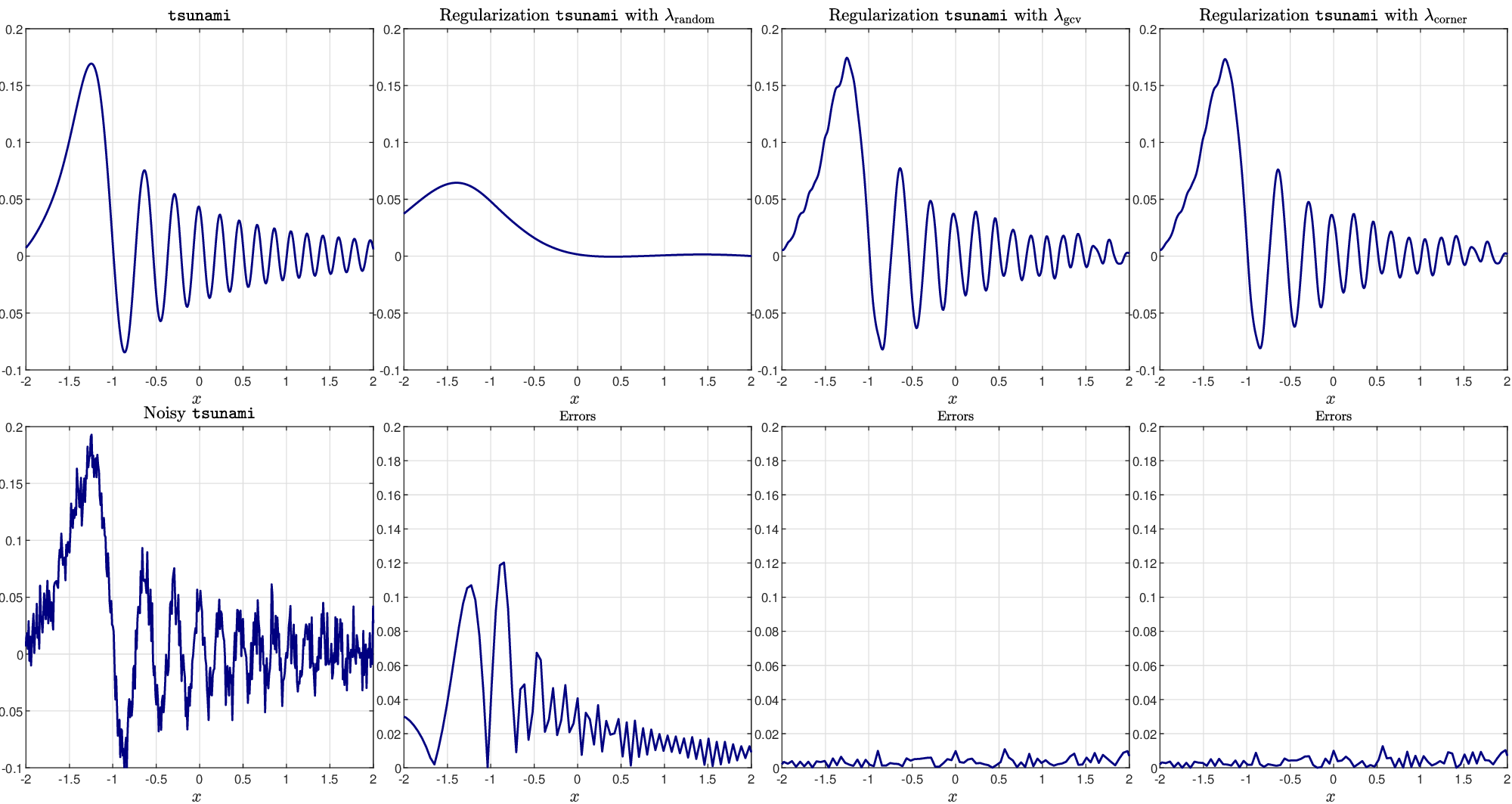}\\
\caption{Recovery efficiency of gallery periodic function  \texttt{'tsunami'} comes from CHEBFUN with 10 dB Gaussian white noise by approximation scheme $p_{\lambda,L,N}^\beta$ \eqref{soltutionp} with parameter $\lambda_{\rm random},\;\lambda_{\rm gcv},\;\lambda_{\rm corner}.$} \label{recover_eps}
\end{figure}

\section{Concluding remarks}
In this paper, we study an $\ell^2_2$-regularized least squares approximation for recovering periodic functions from contaminated data on the unit circle $\mathbb{S}^1$. The factor $\frac{1}{1 + \lambda \beta_{\ell,k}^2}$ in the explicit solution highlights the importance of selecting the regularization parameter. We propose a regularized barycentric trigonometric interpolation scheme under the interpolation condition, providing an innovative technique for recovering periodic functions from noisy data.

The nodes used in the approximation are the $N$-th roots of unity \cite{ExpTrapezoidal}, which can formulate a trapezoidal rule. In future studies, one may relax this trapezoidal rule using the Marcinkiewicz–Zygmund measure \cite{mhaskar2005polynomial,mhaskar2020direct,mhaskar2001spherical,mhaskar2013filtered,mhaskar2017deep}. This relaxation allows for more nodes to be considered as candidates for constructing quadrature rules \cite{An2022exactness,mhaskar2001spherical}, thereby enhancing the approximation.

Numerical examples demonstrate that an appropriate choice of the regularization parameter $\lambda$, tailored to the type of noise, can significantly improve the quality of the approximation. We examine three well-known techniques for various types of noise. Notably, since we have a closed-form solution to \eqref{matrixLeast}, the computation and implementation of the three parameter selection approaches are more straightforward and effective.

The primary advantage of Morozov's discrepancy principle is its ability to manage cases with lower noise levels due to its unique regularity properties. However, this method requires specific assumptions and is sensitive to the actual noise level in the contaminated data. Conversely, the L-curve and GCV methods do not necessitate information about the noise level; rather, they rely solely on the function data $f^{\epsilon}(x_j)$. However, they may be less effective in low-noise scenarios, contrasting with the stability provided by Morozov's discrepancy principle. Therefore, it is crucial to skillfully select the parameter choice strategy based on practical noise information.

\section*{Acknowledgment}
The first author is grateful to Professor Xiaoming Yuan for his helpful comments and suggestions. The first author (C. An) of the research  is partially supported by National Natural Science Foundation of China (No.12371099).

\end{document}